\documentclass[a4paper, numbers, inplaceproofs]{arxiv}
\usepackage{mystyle-mftg-arxiv}

\title[Discrete-Time MFTG: Probabilistic Setup]{Discrete-Time Mean Field Type Games: Probabilistic Setup}
\author[1]{Gr\'egoire Lambrecht}
\author[2]{Mathieu Lauri\`ere}
\affil[1]{Center for Data Science, New York University and NYU Shanghai. \email{gl3048@nyu.edu}}
\affil[2]{Center for Data Science at NYU Shanghai; NYU-ECNU Institute of Mathematical Sciences, NYU Shanghai. \email{mathieu.lauriere@nyu.edu}}
\funding{G.L. is supported by the NSF Award 1922658 and by NYU Shanghai Center for Data Science.}

\begin{document}
\maketitle
\begin{abstract}
We introduce a general probabilistic framework for discrete-time, infinite-horizon discounted Mean Field Type Games (MFTGs) with both global common noise and team-specific common noises. In our model, agents are allowed to use randomized actions, both at the individual level and at the team level. We formalize the concept of Mean Field Markov Games (MFMGs) and establish a connection between closed-loop policies in MFTGs and Markov policies in MFMGs through different layers of randomization. By leveraging recent results on infinite-horizon discounted games with infinite compact state-action spaces, we prove the existence of a stationary Nash-equilibrium profile of closed-loop policies for the original MFTG when the state spaces and the action spaces are compact Polish spaces. We also present an example satisfying our assumptions.
\end{abstract}

\section{Introduction}

Global connectivity gives rise to complex systems across physical and virtual domains~\cite{castells_1996_network}, such as trade networks~\cite{Rodrigue_2017_maritime, maluck_2015_network_network}, air traffic systems~\cite{xu_2017_global_flight}, and social networks \cite{zhang_2012_twitter,achitouv_2024_twitter}. These settings typically involve a large population of interacting agents, making game theory an appropriate modeling framework.
A central concept in game theory is the notion of a Nash equilibrium, introduced by Nash~\cite{nash_1950_equilibrium}, which corresponds to a configuration in which no agent can benefit from a unilateral change in behavior. While the existence of such equilibria is well understood in many settings, solving games numerically becomes infeasible as the number of players increases due to the exponential growth in the number of pairwise interactions. To overcome this challenge, Lasry and Lions~\cite{lasry_2007_mfg} and Huang, Malham{\'e}, and Caines \cite{huangmalhamecaines_2006_mkv} introduced the framework of Mean Field Games (MFGs), drawing inspiration from statistical physics.

MFGs are a powerful framework used to describe the asymptotic behavior of games with homogeneous agents (i.e., agents sharing the same cost function and transition dynamics) as the number of players tends to infinity. The main idea is to reduce the problem to the interaction between a representative agent and the distribution of the population. While MFGs correspond to noncooperative games and focus on Nash equilibria, the cooperative counterpart has been studied under the name Mean Field Control (MFC). MFC is frequently referred to as McKean–Vlasov (MKV) control, since it consists of the joint control of an agent and its law.
MFC can be interpreted as the problem of a central planner who issues instructions to a continuum of agents in order to minimize a social cost.
Both MFGs and MFC have been extensively studied~\cite{huangmalhamecaines_2006_mkv,CarmonaDelarue_book_I,CarmonaDelarue_book_II,CarmonaLauriereTan-2019-mfqrl, carmona_2015_fbsde_mkv,MR3134900, graber_mftc_mfg_2016}.

However, many real-world situations do not involve a single group of homogeneous cooperating agents, but rather the interaction of several teams, each competing to achieve its own goals.
Within a team, all agents are indistinguishable, but the state and action spaces, cost functions, and transition dynamics may differ between teams. Examples include economic competition among groups of companies from different countries and strategic decisions in military conflicts. In situations where agents of the same team cooperate, MFCs have been extended to Mean Field Type Games (MFTGs)~\cite{tembine_2017_mftg,djehiche_2017_mftg}.

MFTGs can be interpreted as a model for very large teams coordinated by central planners. These planners help agents within their team minimize a social cost by first sampling a random policy at the team level, after which each agent samples their own action conditionally on this common team policy. A Nash equilibrium in an MFTG corresponds to an equilibrium between teams, where no team has an incentive to deviate from their policy given the behaviors of other teams. This framework has found applications in modeling crowd motion~\cite{djehiche2017mean}, blockchain interactions~\cite{barreiro2019blockchain}, and electricity prices~\cite{djehiche2020price}. It has recently received growing attention in the mathematics and machine learning literature \cite{cosso_2019_zero-sum,carmona2020policy,subramanian_2023_mftg,guan_2024_zero-sum,sanjari_2024_exchangeable,zaman2024robust}. While algorithms for finding optimal policies have been developed in specific settings~\cite{shao_2024_mftg}, most existing results do not provide general existence theorems for Nash equilibria in the broader MFTG context.

In this work, we study a finite number of cooperative teams, each comprising infinitely many agents, modeled by an MFTG in discrete time with an infinite-horizon discounted objective. Interactions between teams are modeled through the profile of per-team state--action laws.
The dynamics and actions are subject to both idiosyncratic and common noises. In continuous time, the influence of different distributions for idiosyncratic noises on the behavior of MFTGs has been investigated in \cite{ducan_2025_mixture_noises,tembine_2025_rosenblatt}. In contrast, in our work, we do not specify the distribution of the noises and instead propose a general multilevel noise model. To the best of our knowledge, this is the first time that common noise is incorporated into the MFTG framework. The common noise is structured in two levels: a global common noise affecting all teams, and team-specific common noises affecting only the agents within a given team. Certain specific forms of dynamics and cost functions in discrete-time MFTGs have been studied in the literature \cite{barreiro_2025_mftg_high_order_cost}. However, our goal is to propose a very general framework for both the system dynamics and the discounted value function.

A related setting was studied in \cite{sanjari_2024_exchangeable}, where different teams, each driven by decision makers, are in competition. Within each team, the decision makers are cooperative but do not share the same information. A key difference between their model and ours is that our framework includes level-0 agents who follow the central planner of their team. In contrast, \cite{sanjari_2024_exchangeable} considers teams with multiple decision makers and analyzes the mean-field behavior as the number of decision makers tends to infinity. Other studies have considered alternative multi-agent decision-making frameworks, such as \cite{Sanjari2024Optimality}.

For MFC in discrete time with an infinite-horizon discounted objective, \cite{CarmonaLauriereTan-2019-mfqrl} introduced the notion of Mean Field Markov Decision Processes (MFMDPs). They showed that common randomization plays a key role in linking closed-loop policies in MFC to MFMDPs. This reformulation enabled them to leverage existing results from the theory of Markov Decision Processes (MDPs) \cite{bertsekas_1878_book_stocha}. In our study, we extend this reasoning to the MFTG context and use common randomization to link closed-loop policies in the MFTG to a Mean Field Markov Game (MFMG). We refer to this reformulation as the lifted MFMG, since the aggregated state of the players is lifted to the space of probability measures, corresponding to the profile of per-team state laws.
This lifting from a microscopic optimization problem involving agents’ states and actions to a macroscopic formulation in which probability measures themselves serve as the state has attracted significant attention in the recent literature \cite{Sanjari2024Optimality, PhamWei2017McKeanVlasov, Bauerle2023MFMDP, MottePham2022MFMDP}.
This perspective enables us to apply results from the theory of Markov games with state-dependent action sets, where players' available actions may depend on their current state.

However, even though Markov game theory with state-dependent action sets has been extensively studied and understood in many situations, especially when the state and action spaces are finite \cite{altman_1998_constrained_markov_game,altman_1999_book_constrained_mdp}, we did not find results with simple assumptions in the general setting of compact state-action spaces. Even for standard Markov games, the known assumptions ensuring the existence of Nash equilibria remain very restrictive in the compact case \cite{saldi_2024_epsilon_ne}.

Our contributions are as follows:
\begin{enumerate}
    \item We introduce in Section~\ref{section:MFTG_model} a \textbf{general probabilistic framework} for discrete-time, infinite-horizon discounted MFTGs with common noises and interactions through the profile of per-team state--action laws.
    \item We define the notion of a \textbf{lifted MFMG} in Section~\ref{section:MFMG} and analyze its connection to the original MFTG problem when the state and action spaces are compact Polish spaces. We prove the \textbf{equivalence} between solving the MFMG and solving the MFTG (Corollary~\ref{corollary:mftg_mfmg_ne_equivalence}).
    \item We prove in Section~\ref{section:main_result} the \textbf{existence of a Nash equilibrium} in the lifted MFMG when the state spaces and action spaces of the MFTG are compact,
     by relying on a result of Dufour and Prieto-Rumeau for Markov games~\cite{dufour_2024_ne_markov_game}. This consequently establishes the existence of an equilibrium for the MFTG itself.
    \item We illustrate our setting with an \textbf{ARAT} (additive reward, additive transition) \textbf{example} in Section~\ref{section:example}.
\end{enumerate}

The paper is structured as follows.
Section~\ref{section:intuition} presents an informal description of a game that serves as a motivating example for the concept of MFTG.
Section~\ref{section:notation} introduces the notation, terminology, and preliminary results that will be used throughout the paper.
Section~\ref{section:MFTG_model} introduces a general probabilistic framework for discrete-time MFTGs.
Section~\ref{section:MFMG} defines the notion of MFMG.
Section~\ref{sec:relations-models} establishes the connection between MFMG and MFTG policies and value functions.
Section~\ref{section:main_result} applies the theory of Markov games~\cite{dufour_2024_ne_markov_game} to establish the existence of Nash equilibria for MFMGs, and consequently for MFTGs, when the state and action spaces are compact Polish spaces.

\section{Intuition: Finite-population \texorpdfstring{$m$}{m}-team Game}
\label{section:intuition}

In this section, we describe a finite-population game that serves as a motivating example for the concept of an MFTG.
The finite-population discussion and the passage to the infinite-population model are motivational; we do not establish a finite-agent convergence or approximate-Nash theorem in this paper, which focuses on the analysis of the limiting, infinite-population model. The game has $m$ \defi{teams}.
Each team is composed of individuals called \defi{agents} who behave cooperatively within their respective team. All the agents of a team use a common policy chosen by a \defi{central player}. We will sometimes write \defi{player} instead of central player. We use the term ``team'' instead of ``coalition'' as in~\cite{shao_2024_mftg} to avoid confusion with other existing game-theoretic notions.

Let $N^i \in \NN^*$ be the number of individual agents in team $i \in [m]$. At time $n$, agent $j \in [N^i]$ in team $i$ has a state $X^{ij}_n \in \fX^i$ at the $n$-th time step. The central player determines a mixed strategy for the agents in their respective team over some Polish space $\fA^i$. The central player is authorized to choose the strategy randomly. Agent $j$ then selects an action $\alpha^{ij}_n$ according to the strategy assigned to team $i$. For each time step $n$,

let us define for any $i\in [m]$, $\hat a^{i,N^i}_n$, the empirical joint probability measure of the state-action tuples of the team $i$:
$$
    \hat a^{i,N^i}_n = \frac{1}{N^i}\sum_{k \in [N^i]}^{}\delta_{(X^{ik}_n,\alpha^{ik}_n)}, \quad \forall n\in \NN.
$$
The dynamics of the state of agent $k \in [N^i]$ in team $i \in [m]$ are given by:

\begin{equation}
\label{eq:N-agent-transition-ik}
    X^{ik}_{n+1}=F^i\Bigl(X^{ik}_n,\alpha^{ik}_n,\hat a^{1,N^1}_n,\dots,  \hat a^{m,N^m}_n,\varepsilon^{ik}_{n+1}, \varepsilon^{0,0}_{n+1}, \varepsilon^{0,i}_{n+1}\Bigr), \qquad \forall n \ge 0,
\end{equation}

where $F^i$ is a system function and $\varepsilon^{i1}_n,\dots, \varepsilon^{iN^i}_n$ are independent and identically distributed (i.i.d.) random shocks taking values in some Polish space $E^i$. The sequence $(\varepsilon^{0,0}_n)_{n\ge 1}$ represents the global common noise, affecting every agent in every team. For each $i \in [m]$, $(\varepsilon^{0,i}_n)_{n\ge 1}$ represents the common noise of team $i$, affecting only the agents within that team. Moreover, for the transition~\eqref{eq:N-agent-transition-ik}, the agent incurs a cost
$f^i\bigl(X^{ik}_n,\alpha^{ik}_n,\hat a^{1, N^1}_n, \dots, \hat a^{m, N^m}_n \bigr)$, where $f^i$ is a one-step cost function.

The goal of the central player of team $i$ is to find a strategy process $\balpha^i$ to minimize the overall average cost of their team, given by:

$$
    \frac{1}{N^i} \sum_{k=1}^{N^i} \EE\Bigl[\sum_{n=0}^\infty \gamma^n f^i\bigl(X^{ik}_n,\alpha^{ik}_n,\hat a^{1, N^1}_n, \dots, \hat a^{m, N^m}_n \bigr)\Bigr],
$$

where $\gamma \in [0,1)$ is a discount factor. Since the central players interact, one can look for a Nash equilibrium, a situation in which no central player has any incentive to unilaterally change the behavior of their team.

In the present work, we are interested in the asymptotic regime where $N^i\to\infty$ for each $i \in [m]$. Driven by standard propagation-of-chaos results in MFGs, we expect that the states of individual agents become independent in the limit, and that for each $i \in [m]$, every agent of team $i$ evolves according to the following dynamics, written for one representative agent:

$$
    X^{i}_{n+1}=F^i\Bigl(X^{i}_n,\alpha^{i}_n,\PP_{(X^1_n,\alpha^1_n)}^0, \dots, \PP_{(X^m_n,\alpha^m_n)}^0, \varepsilon^{i}_{n+1}, \varepsilon^{0,0}_{n+1},  \varepsilon^{0, i}_{n+1}\Bigr) ,\quad \forall n\ge 0,
$$
where for each $i\in [m]$, $\PP_{(X^i_n,\alpha^i_n)}^0$ is the conditional law of the family of state-action tuples $(X^i_n,\alpha^i_n)$ given the global common noise $(\varepsilon^{0,0}_n)_{n\ge 1}$, the team specific common noise $(\varepsilon^{0,j}_n)_{n\ge 1, j \in [m]}$, and any extra sources of randomness relative to the mixed strategies chosen by the central players. In this limiting scenario, the minimization problem of the central player $i$ is expected to become:
\begin{equation*}
    \inf_{\balpha^i=(\alpha^i_n)_{n\ge 0}} J^i(\balpha^i, \balpha^{-i})
     =
     \inf_{\balpha^i=(\alpha^i_n)_{n\ge 0}} \EE\Bigl[\sum_{n=0}^\infty \gamma^n f^i\bigl(X^i_n,\alpha^i_n,\PP^0_{(X^1_n,\alpha^1_n)}, \dots, \PP^0_{(X^m_n,\alpha^m_n)}\bigr)\Bigr],
\end{equation*}

where $\balpha^{-i} = (\balpha^1, \dots, \balpha^{i-1}, \balpha^{i+1}, \dots, \balpha^m)$ are the actions of the representative agents excluding the $i$-th one and the infima are taken over admissible policies. From here, one can search for the existence of a \defi{Nash equilibrium}, that is, a profile of action processes $\ubalpha = (\balpha^1, \dots, \balpha^m)$ such that
$$
J^i(\balpha^i, \balpha^{-i}) \leq J^i(\balpha', \balpha^{-i}), \quad \forall i \in [m],
$$
for any deviating admissible action process $\balpha'$.
The main goal of the present work is to study the existence of such Nash equilibria.

\section{Notation and Terminology}
\label{section:notation}

\subsection{Usual Spaces}

We denote by $\NN$ the set of natural numbers (including zero), and by $\RR$ the set of real numbers. We define the sets of non-negative and non-zero real numbers as $\RR_+ = \{x \in \RR \mid x \geq 0\}$ and $\RR^* = \RR \setminus \{0\}$, respectively. We combine these notations to obtain $\RR^*_+ = \RR_+ \setminus\{0\}$. Similarly, we denote by $\NN^* = \NN \setminus \{0\}$ the set of strictly positive natural numbers.
For any $N \in \NN^*$, we write $[N] = \{1, \dots, N\}$ to denote the set of the first $N$ positive integers. For any $N^1, N^2 \in \NN$ with $N^1 \leq N^2$, we write $\llbracket N^1, N^2 \rrbracket = \{N^1, N^1+1, \dots,  N^2\}$ to denote the set of integers between $N^1$ and $N^2$. For any set $S$, we denote its power set by $2^S$, that is, $2^S = \{\sS \mid \sS \subseteq S \}$.

\subsection{Game Notation}
Let $\ufX = \fX^1 \times \dots \times \fX^m$ denote the product of sets $\fX^1, \dots, \fX^m$. When the context is clear, we write a generic element $(x^1, \dots, x^m) \in \ufX$ as $\ux = (x^1, \dots, x^m)$. For any index $i$, we use the notation $x^{-i} = (x^1, \dots, x^{i-1}, x^{i+1}, \dots, x^m)$ to denote the tuple excluding the $i$-th component.

Let $\fX$ be a set. For a sequence $(x_n)_{n \in \NN} \in \fX^\NN$, we write $\bx = (x_n)_{n \in \NN}$ and correspondingly $\bm{\fX} = \fX^\NN$.
We will freely combine these notations as needed to define expressions such as $\fX^{-i}, \bm{\fX}, \bm{\fX^{-i}}, \bm{\ux}, \bm{x^{-i}}$.

\subsection{Measurability}
\label{subsec:measurability}
Throughout the paper, we work with \defi{Borel spaces}, namely spaces homeomorphic to a non-empty Borel subset of some Polish space. If $\fX$ is such a space, we denote by $\cB_\fX$ its Borel $\sigma$-field and by $\cP(\fX)$ the space of probability measures on $(\fX, \cB_\fX )$, implicitly assumed to be equipped with the topology of weak convergence and its corresponding Borel $\sigma$-field $\cB_{\cP(\fX)}$. On the product of a finite number of metric spaces $\ufX = \fX^1 \times \dots \times \fX^N$, we consider the product topology, which makes the product again a metric space. If the metric spaces $\fX^i$ are separable, then we have $\cB_{\ufX} = \cB_{\fX^1} \otimes \dots \otimes \cB_{\fX^N}$.

On a Borel space $(\fX, \cB_\fX)$, we write $\delta_{x}$ for the Dirac probability measure at $x \in \fX$, defined on $(\fX, \cB_\fX )$ by $\delta_{x}(B) = \mathds{1}_{x \in B}$ for any $B \in \cB_\fX$.

Let $(\fX, \cB_\fX )$ and $(\fA, \cB_\fA)$ be two measurable spaces. A \defi{stochastic kernel} on $\fA$ given $\fX$ is a mapping $Q: \fX \times \cB_\fA \to \RR_+$ such that $Q_B:x \mapsto Q(B\mid x)$ is measurable on $(\fX, \cB_\fX )$ for every $B \in \cB_\fA$, and $Q_x:B \mapsto Q(B\mid x)$ is in $\cP(\fA)$ for every $x \in \fX$. The mapping $x \mapsto Q(\cdot\mid x)$ can be seen as a random probability measure.

Let $Q$ be a stochastic kernel on $\fA$ given $\fX$. For a bounded measurable function $f:\fA \to \RR$, we denote by $Qf: \fX \to \RR$ the measurable function defined for every $x \in \fX$ by $Qf(x) = \int_{\fA}f(a)Q(da \mid x)$.
For a probability measure $\mu \in \cP(\fX)$, we denote by $\mu \measprod Q$ the unique probability measure on $(\fX\times\fA, \cB_{\fX\times \fA})$ satisfying $(\mu\measprod Q)(B\times C)=\int_B Q(C\mid x)\mu(dx)$ for all $B\in\cB_\fX$ and $C\in\cB_\fA$.

Let $(C,\cF)$ and $(C',\cF')$ be two measurable spaces. The product of the $\sigma$-algebras $\cF$ and $\cF'$ is denoted by $\cF \otimes \cF'$ and consists of the $\sigma$-algebra generated by the measurable rectangles, that is, sets of the form $B \times B'$ for $B \in \cF$ and $B'\in \cF'$. For $\mu$ and $\mu'$, two probability measures on $\cF$ and $\cF'$, we denote by $\mu \otimes \mu'$ the unique probability measure on the product space $(C \times C', \cF \otimes \cF')$ satisfying $\mu\otimes\mu'(B\times B') = \mu(B) \times \mu'(B')$ for all $B \in \cF,B' \in \cF'.$ We give a result that will be useful for what follows:

\begin{lemma}[Blackwell-Dubins Lemma {\normalfont\cite{blackwell_dubins_lemma_1983}}]
    \label{le:BlackwellDubins}
    For any Polish space $B$, there exists a measurable function $\varrho_B:\cP(B)\times [0,1]\to B$, which we shall call the \defi{Blackwell-Dubins function} of the space $B$, satisfying:
    \textbf{(i)} for each $\nu\in\cP(B)$, if $U$ is a random variable uniformly distributed on $[0,1]$, then the $B$-valued random variable $\varrho_B(\nu,U)$ has distribution $\nu$;
    \textbf{(ii)} for almost every $u\in[0,1]$, the function $\nu\mapsto \varrho_B(\nu,u)$ is continuous for the weak topology of $\cP(B)$.
\end{lemma}

Let $\nu \in \cP(\fX)$, and $g:\fX \to \fA$ be a measurable mapping. We denote by $\nu \circ g^{-1}$ the push-forward of the measure $\nu$ by $g$. We use both notations $\PP_\xi$ and $\cL(\xi)$ interchangeably for the distribution of a random element $\xi$.

We say that a function $ f : \fX \times \fA \to C $, where $ (C, d) $ is a metric space, is a \defi{Carathéodory function} if the following two conditions hold: for every $a \in \fA$, the mapping $ x \mapsto f(x, a)$ is measurable on $(\fX,\cB_\fX)$, and for every $x \in \fX$, the mapping $ a \mapsto f(x, a)$ is continuous on $\fA$. The set of such functions is denoted by $\mathrm{Car}(\fX \times \fA, C) $. If, in addition, $f$ is bounded, we write $ f \in \mathrm{Car}_b(\fX \times \fA, C) $.

Given a Borel space $(\fX, \cB_\fX )$ and a probability measure $\lambda \in \cP(\fX)$, we denote by
$$
L^1(\fX, \cB_\fX , \lambda) = \left\{ f:\fX \to \RR \;\bigg|\; \| f\|_{L^1(\fX, \cB_\fX , \lambda)} = \int_{\fX} | f(x) | \lambda(dx) < \infty \right\}
$$
the family of measurable functions (identifying those which are $\lambda$-a.s. equal). Also,
$$
L^\infty(\fX, \cB_\fX , \lambda) = \{f:\fX \to \RR \mid \exists M\in \RR_+ \text{ such that } |f(x)| \leq M \; \; \lambda\text{-a.e.} \}
$$
denotes the set of $\lambda$-essentially bounded measurable functions (again, we identify functions that coincide $\lambda$-a.s.). We will denote by $\|f\|_{L^\infty(\fX, \cB_\fX , \lambda)}$ the corresponding essential supremum. On $L^\infty(\fX, \cB_\fX , \lambda)$, we will consider the \textbf{weak-star} topology (denoted weak-$*$), that is,
$
f_n \xrightharpoonup[n\to\infty]{L^\infty(\fX, \cB_\fX , \lambda)} f
$
if and only if
$
\int_{\fX}f_n(x)h(x)\lambda(dx) \xrightarrow[n\to \infty]{} \int_{\fX}f(x)h(x)\lambda(dx) \quad \text{for each } h \in L^1(\fX, \cB_\fX , \lambda).
$

Let $(\fX, \mathcal{B}_\fX)$ and $(\fA, \mathcal{B}_\fA)$ be two Borel spaces. A \defi{correspondence} from $\fX$ to $\fA$ is a mapping $U$ from $\fX$ to the power set of $\fA$, denoted by $2^\fA = \{\sA \mid \sA \subset \fA \}$, i.e., $U : \fX \to 2^\fA$. We denote it by $U : \fX \twoheadrightarrow \fA$, where for each $x \in \fX$, $U(x) \subset \fA$ is the set of available values associated with $x$. We say that a correspondence is \defi{weakly measurable} if for every $\sA \in \cB_{\fA}$, the set $\{x \in \fX \mid U(x) \cap \sA \neq \emptyset\}$ is a Borel set. Note that, for every measurable mapping $f:\fX\to \fA$, the condition for weak measurability of the correspondence $x \mapsto \{f(x)\}$ coincides with the definition of measurability for the function $f$.

For more details on questions related to measurability, we refer the reader to the textbooks \cite{bertsekas_1878_book_stocha}, \cite{Kallenberg_RM}, and \cite{kallenberg2002foundations}.

\section{\textbf{Definition of MFTG Model and Problem}}
\label{section:MFTG_model}

Before going further, we stress that this work is notationally heavy. Appendix~\ref{app:notation} collects all the notation of the paper into a series of tables, and we recommend that the reader keep it in view while reading.

\subsection{Elements of MFTG}

\begin{definition}[MFTG model]
	\label{def:MFTG_model}
	An infinite-horizon discounted \defi{mean-field-type game (MFTG)} model with common noise, denoted by the tuple
	$(m, \ufX, \ufA, \uE, \uE^{0}, \uF, f^1, \dots, f^m,\eta, \gamma),$
	comprises the following elements:
	\begin{itemize}
		\item A number of teams, $ m \in \NN^* $.

		\item $ \ufX = \fX^1 \times \dots \times \fX^m $ where  $ (\fX^1, \cB_{\fX^1}), \dots, (\fX^m, \cB_{\fX^m}) $ are $ m $ Borel spaces for the state spaces. We denote a generic state profile by $ \ux \in \ufX $.

		\item $ \ufA = \fA^1 \times \dots \times \fA^m $ where $ (\fA^1, \cB_{\fA^1}), \dots, (\fA^m, \cB_{\fA^m}) $ are $ m $ Borel spaces for the action spaces. We denote a generic action profile by $ \ua \in \ufA $.

        \item $ \uE = E^1 \times \dots \times E^m $ where $ (E^1, \cB_{E^1}), \dots, (E^m, \cB_{E^m}) $ are $ m $ Borel spaces corresponding to the idiosyncratic noises.

		\item $\uE^0 = E^{0,0} \times \dots \times E^{0,m}$ where $ (E^{0,0}, \cB_{E^{0,0}}), \dots, (E^{0,m}, \cB_{E^{0,m}}) $ are $ m + 1$ Borel spaces; $E^{0,0}$ corresponds to the space of the global common noise, and for each $i \in [m]$, $E^{0,i}$ corresponds to the space of the common noise affecting players of the team $i$.

		\item $ \uF : \ufX \times \ufA \times \cP(\fX^1 \times \fA^1)\times \cdots \times \cP(\fX^m \times \fA^m) \times \uE \times \uE^0 \to \ufX $ is the joint system function, where the coordinate functions are $ m $ Borel measurable functions given by $ F^i : \fX^i \times \fA^i \times \cP(\fX^1 \times \fA^1)\times \cdots \times \cP(\fX^m \times \fA^m) \times E^i \times E^{0, 0} \times E^{0,i} \to \fX^i $ for each team $ i \in [m] $.

		\item $ m $ bounded Borel measurable functions $ f^i : \fX^i \times \fA^i \times \cP(\fX^1 \times \fA^1)\times \cdots \times \cP(\fX^m \times \fA^m) \to \RR $, for each $ i \in [m] $, called the \emph{one-stage cost functions}.

		\item $\eta$ is a probability measure on $(\cP(\ufX), \cB_{\cP(\ufX)})$ that specifies the law of the initial distribution of the state profile, thereby allowing this distribution itself to be random.

		\item A discount factor $ \gamma \in (0,1) $.
	\end{itemize}
\end{definition}

For each $i \in [m]$, we use the system function $F^i$ to describe the evolution of the states of the representative agent of team $i$. Using system functions allow us to specify the dependency of the system with the different level of noises.
In what follows, we introduce additional spaces to clarify the origins of the different sources of randomness, specify the filtrations associated with shared information, states, and actions, and finally describe how the mean-field interactions are conditioned.

\subsection{Probabilistic Set-up of the MFTG Model}
\label{subsec:proba_set_up}

Building on the intuition provided by the finite-population $m$-team game, we now introduce the precise spaces and distributions associated with the different sources of randomness.

\vskip 5pt
We assume that all the sources of randomness are from a probability space $(\Omega,\cF,\PP)$ supporting:

\begin{enumerate}[label=(\roman*)]
    \item For each team $ i \in [m] $, an i.i.d. sequence of random variables $ \bm{\varepsilon}^i = (\varepsilon^i_{n+1})_{n \geq 0} $, with distribution $ \nu^i \in \cP(E^i) $, modeling the idiosyncratic random shocks. We denote by $ \unu = \bigotimes_{i=1}^m \nu^i $ the product measure over all teams.

    \item An i.i.d. sequence $ \bvarepsilon^{0,0} = (\varepsilon^{0,0}_{n+1})_{n \geq 0} $ of random variables, with distribution $ \nu^{0,0} \in \cP(E^{0,0}) $, modeling the common noise affecting all the players of all the teams.

    \item For each team $ i \in [m] $, an i.i.d. sequence of random variables $ \bm{\varepsilon}^{0,i} = (\varepsilon^{0,i}_{n+1})_{n \geq 0} $, with distribution $ \nu^{0,i} \in \cP(E^{0,i}) $, modeling the common noise affecting players of team $i$. We denote by $ \unu^0 = \bigotimes_{i=0}^m \nu^{0,i} $ the product measure over all the common noises. For each $n \in \NN^*$, we denote by $\uvarepsilon^0_n = (\varepsilon^{0,0}_n, \dots,\varepsilon^{0,m}_n)$ the tuple of the common noises, emphasizing
    that it consists of a total of $m+1$ components (and not $m$).

    \item A random variable $\sU^0$ with distribution $\PP_{\sU^0}$ on a Borel space $(\Upsilon^0, \cB_{\Upsilon^0}) $, providing a source of randomization for the initial distribution.

    \item A random variable $\sU$ with distribution $\PP_{\sU}$ on a Borel space $(\Upsilon, \cB_{\Upsilon}) $ providing a source of randomization for the initial profile of states.

    \item For each team $ i \in [m] $, an i.i.d. sequence $ (\vartheta^i_{n})_{n \geq 0} $ of random variables, taking values in a Borel space $ (\Theta^i, \cB_{\Theta^i}) $, with common distribution $ \PP_{\vartheta^i} $. This sequence provides a source of private randomization for the action choices of the representative agents in team $ i $.

    \item For each team $ i \in [m] $, an i.i.d. sequence $ (\vartheta^{0,i}_{n})_{n \geq 0} $ of random variables, taking values in a Borel space $ (\Theta^{0,i}, \cB_{\Theta^{0,i}}) $, with common distribution $ \PP_{\vartheta^{0,i}} $. This sequence provides a source of randomization for the central player influencing team $ i $.
\end{enumerate}

We let the randomization spaces be abstract for generality. However, for probabilistic construction purposes regarding the Blackwell-Dubins Lemma~\ref{le:BlackwellDubins}, we need to assume some property on the probability measure of these variables.

\begin{assumption}
    \label{assumption:atomless}
    The random sequences $(\bm{\varepsilon}^i)_{i \in [m]}$, $\bvarepsilon^{0,0}$, $(\bm{\varepsilon}^{0,i})_{i \in [m]}$, $\sU^0$, $\sU$, $(\vartheta^i_n)_{n \geq 0, \, i \in [m]}$, and $(\vartheta^{0,i}_n)_{n \geq 0, \, i \in [m]}$ are independent of each other, and the distributions $\PP_{\sU^0}, \PP_{\sU}, \PP_{\vartheta^1}, \dots, \PP_{\vartheta^m},$ and $\PP_{\vartheta^{0,1}}, \dots, \PP_{\vartheta^{0,m}} $ are all atomless.
\end{assumption}

The atomless property is useful since for every Borel space $(B, \cB_{B})$ equipped with an atomless probability measure, there exists a Borel measurable function $h^{B} : B \to [0,1]$ that is uniformly distributed when $B$ is endowed with this measure. Applying this to $\Upsilon^0$, $\Upsilon$, $\Theta^i$, $\Theta^{0,i}$, for each $i \in [m]$, yields Borel measurable functions $h^{\Upsilon^0}$, $h^{\Upsilon}$, $h^{\Theta^i}$, $h^{\Theta^{0,i}}$,  uniformly distributed under $\PP_{\sU^0}$, $\PP_{\sU}$, $\PP_{\vartheta^i}$, $\PP_{\vartheta^{0,i}}$, respectively.

\subsection{Classes of Policies}

\subsubsection{Filtrations, action processes, and control processes}
\label{subsec:filtractions-action-controls}

Several types of randomness have already been identified, corresponding to representative agent level, and central player level. To distinguish these sources of randomness precisely and to formalize how information is shared and actions are chosen, we introduce different filtrations. These filtrations will allow us to clearly define the decision-making processes at each level and to study the resulting mean-field interactions. These are the building blocks to define later the notion of policy (see Paragraph~\S~\ref{subsub:closed-loop_policies}).

The $\sigma$-field for the intial distribution is denoted by $\cF_{\mu_0} = \sigma(\sU^0)$, and the one for the initial profile of states by $\cF_{\ux_0} = \cF_{\mu_0}\vee\sigma(\sU)$.
We introduce some filtration to define the actions and controls processes.
The filtrations of the idiosyncratic noise and the global common noise are denoted by $\cF^{\uvarepsilon}$ and $\cF^{\varepsilon^{0,0}}$ respectively and defined as:
$$
    \cF_0^{\uvarepsilon} = \cF_0^{\varepsilon^{0,0}} = \{ \emptyset, \Omega \},
    \quad
    \cF_n^{\uvarepsilon} = \sigma( \uvarepsilon_1, \dots, \uvarepsilon_n),
    \quad
    \cF_n^{\varepsilon^{0,0}} = \sigma( \varepsilon^{0,0}_1, \cdots, \varepsilon^{0,0}_n ), \quad \  n \geq 1.
$$
For each $i\in[m]$, the filtration of the common noises of team $i$ is:
$
    \cF_0^{\varepsilon^{0,i}} = \{ \emptyset, \Omega \},
    $ $
    \cF_n^{\varepsilon^{0,i}} = \sigma( \varepsilon^{0,i}_1, \dots, \varepsilon^{0,i}_n).
$
The filtration of all the common noises is defined by:
$
    \cF^{\uvarepsilon^0}_n = \cF^{\varepsilon^{0,0}}_n\vee\cdots\vee\cF^{\varepsilon^{0,m}}_n.
$
For each $i\in[m]$, the filtration of the idiosyncratic action randomization of team $i$  is:
$
    \cF_n^{\Theta^i} = \sigma( \vartheta^i_0, \ldots, \vartheta^i_n ),$ $n \geq 0.
$
The filtration of all the idiosyncratic actions randomization is:
$
    \cF_n^{\uTheta} = \sigma( \uvartheta_0, \ldots, \uvartheta_n ), $ $n \geq 0.
$
For each $i\in[m]$, the filtration of the common policies  randomization of team $i$ is:
$
    \cF_n^{\Theta^{0,i}} = \sigma(\vartheta^{0,i}_0, \ldots, \vartheta^{0,i}_n ),$ $n \geq 0.
$
The filtration of all the common policies  randomization is:
$
    \cF_n^{\uTheta^0} = \sigma(\uvartheta^{0}_0, \ldots, \uvartheta^{0}_n ),$ $n \geq 0.
$

The following filtrations represent the information available to the central players and to the level-$0$ agents, for respectively choosing policies and actions. Note that, two of these filtrations depend of the subscript $i \in [m]$.

Let $\FF^0 = (\cF_n^0)_{n \geq 0}$, $\GG^{c,i} = (\cG_n^{c,i})_{n \geq 0}$, $\GG^c = (\cG_n^c)_{n \geq 0}$,
$\GG^{a,i} = (\cG_n^{a,i})_{n \geq 0}$,
$\GG^a = (\cG_n^a)_{n \geq 0}$ defined by
$\cF_0^0 = \sigma(\underline{\vartheta}^0_0 )\vee \cF_{\mu_0}$,
$\cG_0^{c,i} = \cF_{\ux_0} \vee \sigma(\vartheta_0^{0,i} )$, $\cG_0^{c} = \cF_{\ux_0} \vee\sigma(\underline{\vartheta}_0^{0} )$,
$\cG_0^{a,i} = \cF_{\ux_0} \vee \sigma(\vartheta^{0,i}_0, \vartheta_0^i)$,
$\cG_0^a = \cF_{\ux_0} \vee \sigma(\underline{\vartheta}^0_0, \underline{\vartheta_0})$,
and:
\begingroup
\allowdisplaybreaks
\begin{align*}
    & \cF_n^0 = \cF_{\mu_0} \vee \cF_n^{\uvarepsilon^0} \vee \cF_n^{\uTheta^0},
    \quad n \geq 1,
    \\
     &\cG_n^{c,i} = \cF_{\ux_0} \vee \cF_n^{\varepsilon^i} \vee \cF_n^{\uvarepsilon^0} \vee \cF_{n-1}^{\uTheta^0} \vee \cF_{n-1}^{\Theta^i} \vee \cF_{n}^{\Theta^{0,i}}
    ,
    \quad n \geq 1,
    \\
     &\cG_n^c = \cF_{\ux_0}  \vee \cF_n^{\uvarepsilon} \vee \cF_n^{\uvarepsilon^0} \vee \cF_n^{\uTheta^0} \vee \cF_{n-1}^{\uTheta}
    ,
    \quad n \geq 1,
    \\
      &\cG_n^{a,i} = \cF_{\ux_0}  \vee \cF_n^{\varepsilon^i} \vee \cF_n^{\uvarepsilon^0} \vee \cF_{n-1}^{\uTheta^0} \vee \cF_{n-1}^{\Theta^i}
      \vee \cF_{n}^{\Theta^{0,i}}
      \vee \cF_{n}^{\Theta^i}
    ,
    \qquad n \geq 1,
    \\
      &\cG_n^a = \cF_{\ux_0}  \vee \cF_n^{\uvarepsilon} \vee \cF_n^{\uvarepsilon^0} \vee \cF_n^{\uTheta^0} \vee \cF_{n}^{\uTheta}
    ,
    \qquad n \geq 1.
\end{align*}
\endgroup

\label{par:filtrations-discussion}Here, the superscripts $c$ in $\GG^c$ stands for ``control" and $a$ in $\GG^a$ stands for ``action". In Definition~\ref{def:action-control-proc} below $\GG^c$ will be used to define control processes while $\GG^a$ will be used to define action processes below. At time step $n \in \NN$, to pick a control, the central player of the team $i$ uses the $i$-th common randomization $\vartheta^{0,i}_n$. That explains the presence of $\cF^{\Theta^{0,i}}_n$ in  $\cG_n^{c,i}$. To pick an action following this
control, the agent uses their individual randomness $\vartheta^i_n$, and therefore we have $\cG_n^{a,i} = \cG_n^{c,i} \vee \cF_{n}^{\Theta^i}$. The presence of $\cF_n^{\uvarepsilon^0} \vee \cF_{n-1}^{\uTheta^0}$ in $\cG_n^{c,i}$ and not just $\cF_n^{\varepsilon^{0,i}} \vee \cF_{n-1}^{\Theta^{0,i}}$ is justified because the central player will consider the conditional law $\PP^0_{(X^1_n, \dots, X^m_n)}$, which is a $\cF_n^{\uvarepsilon^0} \vee \cF_{n-1}^{\uTheta^0}$-measurable random variable, in order to decide a
control. The filtrations $\GG^c$ and $\GG^a$ are introduced to characterize the measurability of a profile of controls and a profile of actions.

\label{par:information-in-words}We now describe in words the information available to each decision maker.
At time $n$, the central player of team $i$ knows: the initial data
($\cF_{\ux_0}$); the idiosyncratic noises of its \emph{own} team up to time $n$
($\cF_n^{\varepsilon^i}$), all common noises up to time $n$ ($\cF_n^{\uvarepsilon^0}$), the common
randomizations of \emph{all} teams up to time $n-1$ ($\cF_{n-1}^{\uTheta^0}$)
together with its own current randomization $\vartheta_n^{0,i}$
($\cF_n^{\Theta^{0,i}}$), and its own team's individual randomizations up to
$n-1$ ($\cF_{n-1}^{\Theta^i}$). Two exclusions are essential. First, planner $i$
does \emph{not} observe $\bm{\varepsilon}^j$ for $j \neq i$. Second, planner $i$ does not observe
the current randomizations $\vartheta_n^{0,j}$ of the other planners, only their
past values: the planners therefore randomize simultaneously and independently at
each time step.

A level-$0$ agent of team $i$ has access to everything its planner does, together
with its own private randomization $\vartheta_n^i$, whence
$\cG_n^{a,i} = \cG_n^{c,i} \vee \cF_n^{\Theta^i}$. Finally, we stress that
$\FF^0$, $\GG^c$ and $\GG^a$ are \emph{not} the information of any decision maker.
$\FF^0$ is the common information generated by all common noises and common
randomizations, with respect to which the mean-field conditional laws $\PP^0$ are
taken.

\vskip 6pt
We now describe how a central player and a generic agent choose their actions and policies using the filtrations defined above. We will refer to the \emph{level-0} framework for anything related to a generic agent. In Section~\ref{section:MFMG}, we will introduce a lifted stochastic optimization model, whose elements will be referred to as \emph{level-1}. For reasons that will become clear later, the representative central player of a team will be called the \emph{level-1 player} of this team.

\begin{definition}
    \label{def:action-control-proc}
    For each $i \in [m]$, a \defi{level-0 action associated to team $\bm{i}$} is an element of $\fA^i$. A \defi{level-0 (mixed) control associated to team $\bm{i}$} is a random probability measure on $(\fA^i, \cB_{\fA^i})$, that is, a random variable with values in the Borel space $(\cP(\fA^i), \cB_{\cP(\fA^i)})$. A \defi{level-0 action process associated to team $\bm{i}$} is a sequence of random variables $\balpha^i = (\alpha^i_n)_{n \geq 0}$ with values in $\fA^i$ which is adapted to the filtration $\GG^{a, i}$. The set of such action processes is denoted by $\AA^i$. A \defi{level-0 control process associated to team $i$} is a sequence $\bfa^i = (\fa^i_n)_{n \geq 0}$ of level-0 controls which is adapted to the filtration $\GG^{c,i}$.
    Finally, an  action process $\balpha^i = (\alpha_n^i)_{n \geq 0}$ is said to be a \defi{realization} of a level-0 control process  $\bfa^i = (\fa_n^i)_{n \geq 0}$ if
    $\cL \big( \alpha_n^i \, | \, \cG_n^{c, i} \big) = \fa_n^i$,
    $\PP-a.s.$,  for every $n \geq 0$.  We extend naturally all these definitions to a \defi{profile of level-0 actions}, a \defi{profile of level-0 (mixed) controls}, a \defi{profile of level-0 action processes}, and a \defi{profile of level-0 control processes}.
\end{definition}

Intuitively, an action process is the realization of a control process, where the sampling is performed using $(\vartheta^i_n)_{n \ge 0}$. We use the term \emph{mixed} for the associated probability measures, and \emph{randomized} for the corresponding random variables. A generic level-0 action of the $i$-th central player will be denoted by $a^i$, while a generic level-0 control will be denoted by $\fa^i$. Even if we do not specify it, all level-0 controls considered are implicitly assumed to be mixed.

By definition of a realization, for any level-0 control process $\fa^i$ and any two realizations $\balpha^i,\balpha^{' i}$ of $\fa^i$, for every bounded Borel measurable function $h$ we have
$\EE \left[ h( \alpha_n^{'i} ) \, | \, \cG_n^{c,i} \right] = \int_{\fA^i} h(a^i) \fa_n^i(d a^i) = \EE \left[ h( \alpha^i_n ) \, | \, \cG_n^{c,i} \right]$, $\PP$-a.s., $n \ge 0$.

\subsubsection{Conditional distribution and state process}\,

We can now describe the mean-field interactions and show how the dynamic of the states processes are driven by mixed control processes, affected by common noises.

\begin{definition}
    \label{def:MFTG-state_process_from_control_process}
    For any profile of level-$0$ action processes $\ubalpha$, we say that a profile of
    processes $\ubX^{\ubalpha} = (\bX^{\ubalpha,1}, \dots, \bX^{\ubalpha,m})$ is a
    \defi{profile of state processes associated to} $\ubalpha$ for the MFTG model if the
    following three conditions hold.
    \begin{enumerate}[label=\roman*)]
        \item The random initial distribution is
        $
            \mu_0 := \varrho_{\cP(\ufX)}\bigl(\eta,\, h^{\Upsilon^0}(\sU^0)\bigr),
        $
        where $\varrho_{\cP(\ufX)}$ is the Blackwell--Dubins function of $\cP(\ufX)$ given
        by Lemma~\ref{le:BlackwellDubins}. It is a $\cP(\ufX)$-valued,
        $\cF_{\mu_0}$-measurable random variable with distribution $\eta$.

        \item The initial profile of states is
        $
            \uX_0^{\ubalpha} := \varrho_{\ufX}\bigl(\mu_0,\, h^{\Upsilon}(\sU)\bigr),
        $
        where $\varrho_{\ufX}$ is the Blackwell--Dubins function of $\ufX$. It is an
        $\ufX$-valued, $\cF_{\ux_0}$-measurable random variable satisfying
        $
            \cL\bigl(\uX_0 \mid \cF_{\mu_0}\bigr) = \mu_0,
        $
        $\PP$-a.s.

        \item For every $n \geq 0$ and every $i \in [m]$,
        \begin{equation}
            \label{eq:MFGT-system_dynamics_level_0}
            X_{n+1}^{\ubalpha, i} = F^i\big( X_n^{\ubalpha, i}, \alpha^{i}_n,
            \PP^0_{(X_n^{\ubalpha, 1}, \alpha_n^1)}, \dots, \PP^0_{(X_n^{\ubalpha, m}, \alpha_n^m)}, \varepsilon^{i}_{n+1}, \varepsilon^{0,0}_{n+1} ,
            \varepsilon^{0,i}_{n+1} \big),
        \end{equation}
        which can be rewritten as
        \begin{equation}
            \label{eq:MFGT-system_dynamics_level_0_full}
            \uX_{n+1}^{\ubalpha} = \uF\big( \uX_n^{\ubalpha}, \ualpha_n,
            \PP^0_{(X_n^{\ubalpha, 1}, \alpha_n^1)}, \dots, \PP^0_{(X_n^{\ubalpha, m}, \alpha_n^m)}, \uvarepsilon_{n+1},
            \uvarepsilon^{0}_{n+1} \big),
        \end{equation}
        where for each $i \in [m]$, $\PP^0_{(X_n^{\ubalpha, i}, \alpha_n^i)}$ denotes a regular version of
        $\cL\bigl( (X_n^{\ubalpha, i}, \alpha_n^i) \mid \cF_n^0 \bigr)$ , the conditional
        joint distribution of the family of state--action pairs, for the team $i$, at time $n$ given the
        initial distribution, the common noises and the common randomizations.
    \end{enumerate}
\end{definition}

\begin{remark}
    In the mean field control and mean field game literature, the initial
    configuration of the population is usually given by a fixed distribution
    $\mu_0 \in \cP(\ufX)$. We allow instead $\mu_0$ to be random with law $\eta$,
    for two reasons. First, in many situations the initial configuration is not
    known exactly. Second, this generality is free: none of the assumptions
    of Section~\ref{section:main_result} bears on $\eta$, so the fixed-initial-
    distribution setting is recovered by taking $\eta = \delta_{\mu_0}$.
\end{remark}

As for $\PP^0_{(X_n^{\ubalpha, i}, \alpha_n^i)}$, for each profile of level-$0$ action
processes $\balpha^i$ and each $n \geq 0$, we denote by $\PP^0_{X_n^{\ubalpha, i}}$ a
regular version of the conditional distribution
$\cL(X_n^{\ubalpha, i}\,|\,\cF^0_n)$. By the recursive construction and Assumption~\ref{assumption:atomless}, $\uvartheta_n^0$ is independent of $\sigma(X_n^{\ubalpha,i})\vee\cF_{\mu_0}\vee \sigma(\uvarepsilon^0_k,\underline{\vartheta}^{0}_{k-1},\,1\leq k\leq n)$. Consequently,
\begin{equation}
\label{fo:P0X_n}
    \PP^0_{X_n^{\ubalpha, i}}
    = \cL\Bigl(X_n^{\ubalpha, i}\,\Big|\,\cF_{\mu_0} \vee {}
    \sigma(\uvarepsilon^0_k,\underline{\vartheta}^{0}_{k-1}, \, 1 \leq k \leq n)\Bigr),
    \qquad \PP\text{-a.s.}.
\end{equation}

\subsubsection{Closed-loop policies}\label{subsub:closed-loop_policies}

We now introduce the concept of closed-loop policies. These policies are the focus of our study because, unlike open-loop policies, they are more commonly used in practical applications due to their relative ease of implementation.

\begin{definition}\label{def:Markovian_policy}
    For each team $i \in [m]$ we say that $\pi^i$ is a \defi{closed-loop Markov strategy function associated to team $\bm{i}$} if it is a measurable function from
    $\fX^i \times \cP(\fX^1) \times \dots \times \cP(\fX^m) \times \Theta^{0,i}$ into $\cP(\fA^i)$.
    A \defi{closed-loop Markov policy associated to team $\bm{i}$}  is a sequence $\bpi^i=(\pi^i_n)_{n\ge 0}$  of such functions.
    The set of all closed-loop Markov policies associated to team $i$ is denoted by $\bPi^{i}$.
    We naturally extend these definitions to a \defi{profile of level-0 close-loop policies} with a generic element denoted $\ubpi$, a \defi{ profile of level-0 close-loop strategy functions} with a generic element denoted $\underline{\pi}$, and the set of all the profiles of level-0 closed-loop policies denoted $\underline{\bPi}$. We say that a profile of level-$0$ policies $\ubpi=(\bpi^i)_{i\in[m]}$ is \defi{stationary} if, for every $i\in[m]$, the strategy functions $\pi_n^i$ are equal for all $n$.
\end{definition}

The Markov policies are chosen in this form because the dynamics~\eqref{eq:MFGT-system_dynamics_level_0} and the cost~\eqref{eq:MFTG-J_alpha} depend on both the state and the mean-fields. Since we work with mixed strategies, a policy associated with team $i$ takes values in $\cP(\fA^i)$. This highlights that an action $\alpha_n^i \in \fA^i$, drawn according to such a policy, is sampled from a probability measure that depends directly on the values of $X_n^{\balpha, i}$, the laws $\PP^0_{X_n^{\balpha, 1}}$,\dots, $\PP^0_{X_n^{\balpha, m}}$ and the random variable $\vartheta_n^{0, i}$. The following definition formalizes this mechanism.

\begin{definition}\label{def:admissible_state_action_processes}

	For a profile of closed-loop Markov policies $\ubpi\in \underline{\bPi}$, a pair consisting of a state-profile process and an action-profile process $(\ubX, \ubalpha)$ is said to be \defi{generated by}
    $\ubpi$
    if:

	\begin{enumerate}[label=\roman*)]
		\item $\ubX$ is a state process associated to
        $\ubalpha$
        in the sense of  Definition~\ref{def:MFTG-state_process_from_control_process}.
		\item For any $i \in [m]$ the action process $\balpha^i$ is adapted to $\GG^{a,i}$ and satisfies

        \begin{equation}
			\label{eq:definition_action_Markov_closed_loop}
			\cL \big( \alpha^i_n \, | \,  \cG_n^{c,i} \big) = \pi^i_n \big(  X^i_n, \, \PP^0_{X_n^1}, \dots, \PP^0_{X_n^m},\, \vartheta_n^{0,i} \big), \quad \PP-a.s., \quad n \geq 0, \, \forall \, i \in [m].
		\end{equation}

        Notice that for each $i, j \in [m], i \ne j$, we have $\cG^{a,i}_n \perp_{\cG^c_n} \cG^{a,j}_n$, therefore $\alpha_n^i \perp_{\cG^c_n} \alpha^j_n$, and~\eqref{eq:definition_action_Markov_closed_loop} is equivalent to the following:

        \begin{equation}
        \label{eq:definition_action_Markov_closed_loop_full}
			\cL \big( \ualpha_n \, | \,  \cG_n^c \big)(d\ualpha) = \prod_{i = 1}^m \pi^i_n \big(  X^i_n, \, \PP^0_{X_n^1}, \dots, \PP^0_{X_n^m},\, \vartheta_n^{0,i} \big)(d\alpha^i) , \quad \PP-a.s., \quad n \geq 0.
		\end{equation}
	\end{enumerate}
\end{definition}

The state and action processes are constructed simultaneously, using the system dynamics~\eqref{eq:MFGT-system_dynamics_level_0_full} and the sampling procedure with~\eqref{eq:definition_action_Markov_closed_loop_full}.

As described in \cite{CarmonaLauriereTan-2019-mfqrl}, a convenient way to construct an action process $\ubalpha$ satisfying \eqref{eq:definition_action_Markov_closed_loop_full} is to use the Blackwell-Dubin's lemma (Lemma~\ref{le:BlackwellDubins}).

\begin{remark}
    As noticed in~\cite{CarmonaLauriereTan-2019-mfqrl}, even though we call a policy $\bpi\in \bPi$ a ``Markov" policy, it does not imply any Markov property for the state process $\bX$ associated to such a policy. This abuse of terminology can be explained by our intention to work with level-1 Markov policies which will imply the Markov property for a lifted measure-valued state process constructed in the next section.  Also, since we only use the term ``Markov policy'' in the closed-loop setting, we shall most often drop the term Markov hereafter and only call them simply closed-loop policies.
\end{remark}

\subsection{Value Functions and Equilibrium}
\label{subsection:value-function}

To complete the definition of a MFTG problem with closed-loop policies, we introduce the value function and the associated optimization problem.

\begin{definition}For any $i \in [m]$, the \defi{value function associated to team $\bm{i}$}, when the central players choose a profile $\ubalpha$ of level-0 action processes is defined as
	\begin{equation}
		\label{eq:MFTG-J_alpha}
			J^{i}(\balpha^i, \balpha^{-i}) := \EE \Big[\sum_{n \geq 0} \gamma^n f^i\left(X_n^{\ubalpha, i}, \alpha^i_n,\PP^0_{(X_n^{\ubalpha, 1}, \alpha_n^1)}, \dots, \PP^0_{(X_n^{\ubalpha, m}, \alpha_n^m)} \right)  \Big],
	\end{equation}
	where $\ubX^{\ubalpha}$ is associated to $\ubalpha$ according to the dynamics~\eqref{eq:MFGT-system_dynamics_level_0}.
\end{definition}

Note that since $f^i$ is measurable and bounded, the value function $J^{i}(\balpha^i, \balpha^{-i})$ is well-defined for every $(\balpha^i, \balpha^{-i})$. This value depends only upon the sequence of joint distributions of the
$\cP(\fX^j\times\fA^j)$-valued processes
$(\PP^0_{(X_n^{\ubalpha, j},\alpha_n^j)} )_{n\ge 0}$, $j \in [m]$. For any closed-loop policy $\ubpi$, we can show that
$\PP^0_{(X_n^{\ubalpha, j},\alpha_n^j)}$ depends on the action process only through the policy $\ubpi$ provided $(\ubX,\ubalpha)$ is generated by $\ubpi$;
see~Lemma~\ref{le:level1_pair_determined_by_policy}, proved in Appendix~\ref{app:level1-pair-determined}. As a consequence, we can define, for any level-$0$ action process $\balpha$  generated by $\bpi$,
$
        J^{i }(\ubpi) = J^{i}(\ubalpha),
$ for each $i \in [m]$.

We now introduce the central definition in an MFTG problem: that of a Nash equilibrium.

\begin{definition}\label{Nash-equilibrium-profile}
    A profile of level-0 closed-loop policies $\ubpi^* \in \ubPi$ is a level-$0$ Nash equilibrium if:
        \begin{equation*}
            J^{i}(\bpi^{*,i}, \bpi^{*,-i}) \leq J^{i}(\bpi^i, \bpi^{*,-i}), \quad \forall \ubpi \in \ubPi, \,\forall i \in [m].
        \end{equation*}
\end{definition}

We are now about to state the principal result of our paper. To prove this result, we will introduce a new game in which the state is the profile of the laws of the states of the teams, which can be interpreted as the profile of mean fields. The assumptions are mainly related to this new game, and are therefore detailed in the next sections.
\begin{theorem}
    \label{thm:Main_result}
    Suppose that Assumptions~\ref{assumption:atomless},~\ref{assumption:ufX-ufA-compact},~\ref{assumption:continuity}, and~\ref{assumption:continuity_sY} hold.
    Then, for any distribution $\eta \in \cP(\cP(\ufX))$ of initial distributions, the MFTG admits a stationary level-$0$ Nash equilibrium associated with $\eta$, in the sense of Definition~\ref{Nash-equilibrium-profile}.
\end{theorem}
In this result, the assumptions do not concern the distribution $\eta$, and we can therefore recover the existence of a stationary level-$0$ Nash equilibrium for any deterministic initial distribution $\mu_0 \in \cP(\ufX)$ by taking $\eta = \delta_{\mu_0}$.

The dependence of the dynamics $\uF$, the costs $\underline{f}$, and the strategies on the profile of per-team state--action laws
makes the problem more complex than in classical game settings. To simplify the analysis of MFTG problems, we aim to avoid working with
an augmented state space consisting of both the agent's state and the profile of per-team state--action laws.
Instead, we introduce a new classical game and then show how results from game theory can be used to establish that this new game admits a Nash equilibrium. This result, in turn, implies the existence of an equilibrium in the original MFTG problem.

\section{\textbf{Mean Field Markov Game}}
\label{section:MFMG}

Mean Field Markov Games (MFMG) play the same role for MFTG as Mean Field Markov Decision Processes (MFMDP) do for Mean Field Control problems, as introduced in~\cite{CarmonaLauriereTan-2019-mfqrl}. The underlying idea is strictly analogous: to model a game in which each player is the central planner of a team. To achieve this, idiosyncratic noises and individual randomizations  are averaged out according to their laws, so that only the common noises and common randomizations remain in the model. In this framework, the state of the players is collectively represented by an element of
$\cP(\fX^1)\times\cdots\times\cP(\fX^m)$, and each player $i \in [m]$ selects actions in
$\cP(\fX^i\times \fA^i)$. The action space is thus
$\cP(\fX^i \times \fA^i)$ and not $\cP(\fA^i)$.

\subsection{Mean-field Markov Game Framework}

\begin{definition}
    \label{def:MFMG}
A  \defi{mean-field Markov game} consists of a tuple

$$
(m, S^1, \dots, S^m, A^1, \dots, A^m, \hat U^1, \dots , \hat U^m, P, \hat f^1,\dots, \hat f^m, \hat\eta,\gamma),
$$

as described below:
	\begin{itemize}
        \item A number of players, $m \in \NN^*$.

	    \item  The compact Polish space
        $S^i$, with $i\in [m]$, which serves as the underlying space measured by the mean field of team $i$. We consider the standard Borel space
        $(S^i, \cB_{S^i})$. The actual state space of the game is then given by
        $\hatuS:= \hat S^1\times \cdots\times \hat S^m := \cP(S^1)\times \cdots\times\cP(S^m)$.
        $\hatuS$ represents the state space of the players; a generic element in
        $\hatuS$ is denoted by
        $\umu = (\mu^1, \dots, \mu^m)$.

        \item The compact Polish space $A^i$, with $i \in [m]$. We consider the standard Borel space  $(A^i, \cB_{ A^i})$ associated to $A^i$. The actual action space of team $i$ is then given by the space of joint probability measures
        $\hat A^i := \cP(S^i \times A^i)$. A generic element of $\hat A^i$ is denoted $\hat a^i$.
        We define the profile of action space $\hatuA := \hat A^1 \times \cdots \times \hat A^m$, endowed with the product topology. A typical element of $\hatuA$ will be written $\hatua = (\hat a^1, \dots, \hat a^m)$.

        \item The continuous correspondence $\hat U^i: \hat S^i \twoheadrightarrow \hat A^i$, with $i \in [m]$, stands for the measurable mapping giving the available actions of player $i$. Given any $\mu^i \in \hat S^i$, the nonempty, measurable, convex and compact set $\hat U^i(\mu^i) \subset \hat A^i$ is the set of actions available to player $i$ at state $\umu$. It is defined by:
        \begin{equation}
            \hat U^i(\mu^i) := \left\{ \hat a^i \in \hat A^i, \text{ such that } \operatorname{pr_{s^i}}(\hat a^i) = \mu^i \right\},
        \end{equation}
        where $\operatorname{pr_{s^i}}$ is defined naturally on every $\hat A^i$ by
        \begin{equation*}
        \arraycolsep=1pt\def\arraystretch{1.5}
		\begin{array}{lll}
            \operatorname{pr_{s^i}} & : \hat A^i& \longrightarrow \hat S^i \\
              & \hat a^i & \longmapsto (ds \mapsto \int_{A^i} \hat a^i(ds^i, da^i))
        \end{array}
        \end{equation*}
         and $\hatuU(\umu) := \hat U^1(\mu^1) \times \cdots \times \hat U^m(\mu^m)$.
		Let us define the set of admissible states and profile of actions:
		\begin{equation}
			\label{eq:Gamma_bar}
			\Sigma  := \left\{(\umu,\hatua)\in{(\hatuS\times{\hatuA})};\; {\hat a^i}\in{\hat U^i}(\mu^i), \, i \in [m] \right\}.
		\end{equation}

		\item The stochastic kernel $P: \hatuS\times \hatuA \times \cB_{\hatuS} \to \RR_+$. It stands for the transition kernel.
        \item Given $i \in [m]$, the bounded measurable function
        $\hat f^i: \hatuS \times \hatuA \to \RR$. It stands for the one-stage cost function of player $i$.

        \item An initial probability measure $\hat\eta \in \cP(\hatuS)$, the law of the initial state.
		\item The discount coefficient $\gamma \in (0,1)$.
	\end{itemize}
\end{definition}

\begin{remark}
    \label{rmk:prx_Ui}
	We note that \cite[Remark 13]{CarmonaLauriereTan-2019-mfqrl} extends to our setting. The projection map
    $\operatorname{pr_{s^i}}$ is continuous, so the constraint set $\hat U^i(\mu^{i})$ is closed in $\hat A^i$ and measurable for every
    $\mu^i \in \hat S^i$ (the correspondence $\hat U^i$ is weakly measurable). The graph
    $Gr(\pr_{s^{i}}) := \{ (\mu^i, \hat a^i): \pr_{s^i}(\hat a^i) = \mu^i\} \subset \hat S^i \times \hat A^i$ is closed, so $\Sigma $ is also closed in
    $\hatuS\times\hatuA$. Hence $\Sigma $ is an analytic subset of
    $\hatuS\times\hatuA$, and a Polish space on its own. We assume that $\Sigma $ is endowed with the induced topology as well as the trace $\sigma$-field inherited from
    $\hatuS\times\hatuA$.
\end{remark}

\subsection{MFMG Markov Policies}\,
\vskip 5pt

By analogy with Definition~\ref{def:Markovian_policy} in MFTG, we now define the notions of mixed strategy and mixed Markov policy for MFMGs.

\begin{definition}\label{def:MFMG_POLICIES}
    We call \defi{level-1 mixed strategy function associated to player $\bm{i}$} any Borel measurable function $\hat\pi$ from
    $\hatuS$ into $\cP(\hat A^i)$ satisfying

        $
            \hat \pi(\umu) (\hat U^i(\mu^i) )  = 1, \, \umu \in \hatuS.
        $

    We denote by $\hat\Pi^i$ the set of mixed strategy functions.
    A \defi{mixed Markov policy associated to player $\bm{i}$} is an element of $\hat\bPi^i := (\hat\Pi^i)^{\NN}$.
    We say that a policy $\hat\bpi=(\hat\pi_n)_{n\ge 0}$ is \defi{stationary} if the strategy functions $\hat \pi_n$ are equal for all $n$.
    We extend naturally these definitions to a \defi{profile of level-1 mixed strategy functions}, a \defi{profile of mixed Markov policies}. We denote by $\hatuPi$ the set of
    mixed profiles of strategy functions. We denote by $\hatubPi$ the set of mixed profiles of policies.
\end{definition}

For any profile of strategies $\hatupi \in \hatuPi$, we will write with an abuse of notations, for any
$\umu \in \hatuS$,
$\hatupi(\umu)$ for the product probability measures over $\hatuA$ obtained from $\hat \pi^1, \dots, \hat\pi^m$:

\begin{equation}
    \label{eq:hatupi_prod}
    \hatupi(\umu)(d\hatua) =\prod_{i=1}^m \hat\pi^i(\umu)(d\hat a^i), \quad \forall \umu \in \hatuS.
\end{equation}

In the spirit of the previous section, these policies should be called ``Markov'' policies.

Before stating how a sequence of state-actions is generated in the MFMG, we define for each $i\in[m]$, the operator $\operatorname{pr_{s^i}}: \cP(\uS)\to \cP(S^i)$ for every $\mu \in \cP(\uS)$ by its $i$-th marginal $\operatorname{pr_{s^i}}(\mu)(ds^i) = \int_{S^{-i}}\mu(ds^1, \dots, ds^m)$. This operator was already defined on $\hat A^i$, but this abuse of notation should not lead into any confusion.

\begin{definition}
    \label{def:MFMG_admissible_state_action_processes}
    Let $\hatubpi \in \hatubPi$ be a profile of policies. Let $(\ubmu,\hatuba ) = ((\umu_n)_{n \ge 0}, (\hatua_n)_{n \ge 0})$ be two sequences of random variables taking values in the Borel spaces $(\hatuS, \cB_{\hatuS})$ and $(\hatuA, \mathcal{B}_{\hatuA})$. We say that $(\ubmu,\hatuba)$ is \defi{generated by} $\hatubpi$ if and only if the following conditions hold:
    \begin{align}
        \cL(\umu_0) &= \hat\eta, \\
        \cL(\umu_{n+1}\vert \umu_n, \hatua_n) &= P(\umu_n, \hatua_n), \quad &\forall n \ge 0, \label{eq:level-1_gen_law_mun} \\
        \cL(\hatua_n\mid \umu_n) &= \hatupi_n(\umu_n),
        \quad \forall n \ge 0, \label{eq:level-1_gen_law}
    \end{align}
    where we recall the notation in~\eqref{eq:hatupi_prod}.
\end{definition}

\subsection{Value Functions and Equilibrium for MFMG}\,

The value functions for the MFMG are defined in a usual way.

\begin{definition}
    \label{def:MFMG_VALUE_FUNCTION}
	For any $i \in [m]$, the \defi{value function associated to player $\bm{i}$} in the MFMG, when the players use policy profile $\hatubpi$ is
	\begin{equation}
		\label{eq:MFMG_valuef}
			\hat J^{i}(\hat \bpi^i, \hat \bpi^{-i}) := \EE \left[\sum_{n \geq 0} \gamma^n \hat f^i(\umu_n, \hatua_n)  \right],
	\end{equation}
	where the pair of state-profile of action processes  $(\bm{\umu}, \hatuba)$ is generated by $\hatubpi$.
\end{definition}

It can be shown that the value functions $\hat J^{i}(\hatubpi)$ given in~\eqref{eq:MFMG_valuef} is well defined because the expectation in~\eqref{eq:MFMG_valuef} does not depend upon the particular choice of the pair of state action processes $(\bm{\umu},\hatuba)$ generated by $\hatubpi$.
See Appendix~\ref{app:MFMG_valuef}.
We now give the definition of a Nash equilibrium for the MFMG.

\begin{definition}\label{Nash-equilibrium-level-1}
    \label{def:MFMG_NASH}
    We say that a profile of policies $\hatubpi^* \in \hatubPi$ is a Nash equilibrium if:
    \begin{equation*}
            \hat J^{i}(\hatubpi^{*,i}, \hatubpi^{*,-i}) \leq \hat J^{i}(\hatubpi^i, \hatubpi^{*,-i}), \quad \forall \hatubpi \in \hatubPi, \,\forall i \in [m].
    \end{equation*}
\end{definition}

\subsection{MFMG Lifted from MFTG}

From a given MFTG, we can define a natural MFMG, where the shared state corresponds to the mean fields of the MFTG.
Definition~\ref{def:MFMG} requires the state spaces and action spaces to be compact, for this reason we also assume that this assumption holds for the state spaces and action spaces of the MFTG.

\begin{assumption}[Compactness]
    \label{assumption:ufX-ufA-compact}
    For each $i \in [m]$, the spaces $\fX^i$ and $\fA^i$ are compact.
\end{assumption}

\begin{definition}\label{def:MFTMDP_problem}
    The tuple
    $
    (m, S^1, \dots, S^m,$ $A^1, \dots, A^m,$ $\hat U^1, \dots, \hat U^m,$ $P, \hat f^1,\dots, \hat f^m, \hat\eta, \gamma)
    $
    is said to be the \defi{MFMG lifted from the MFTG model} $
    (m, \ufX, \ufA, \uE, \uE^{0}, \uF, f^1, \dots, f^m, \eta,\gamma)
        $
    of Definition~\ref{def:MFTG_model} if it satisfies:
    \begin{itemize}
        \item For each $i \in [m]$, $S^i = \fX^i$, and therefore $\hat S^i = \cP(\fX^i)$. We also have $\hatuS = \hat S^1\times\cdots\times\hat S^m$.
        \item The initial distribution is $\hat\eta := \eta\circ\mathfrak m^{-1} \in\cP(\hatuS)$, where $\mathfrak m:\cP(\ufX)\to\hatuS$ is the marginalization map defined by $\mathfrak m(\mu) := \bigl(\operatorname{pr_{x^i}}(\mu)\bigr)_{i\in[m]}$.
        \item For each $i \in [m]$, $A^i = \fA^i$, and therefore
        $\hat A^i = \cP(\fX^i\times\fA^i)$. We also have
        $\hatuA = \hat A^1\times \cdots\times \hat A^m$.
        \item The transition kernel $P$ is given by

        \begin{equation}
        \label{eq:MFMG_transition_kernel_from_system_func}
            P( \umu, \hatua)(d \umu') = \big( \unu^0 \circ \hat F( \umu, \hatua, \cdot)^{-1}\big)(d \umu'), \qquad (\umu, \hatua) \in \hatuS\times\hatuA ,
        \end{equation}
        where, writing $\ue^0 = (e^{0,0}, e^{0,1}, \dots, e^{0,m}) \in \uE^0$, the map $\hat F$ is defined componentwise by
        \begin{equation}
        \label{eq:MFMGM_Fbar_pushforward}
            {\hat F}(\umu, \hatua ,\ue^0) := \bigl( {\hat F}^1(\umu, \hatua ,\ue^0), \dots, {\hat F}^m(\umu, \hatua ,\ue^0) \bigr) \in \hatuS,
        \end{equation}
        with, for each $i \in [m]$,
        \begin{equation}
        \label{eq:MFMGM_Fbar_pushforward_i}
            {\hat F}^i(\umu, \hatua ,\ue^0)  =   (\hat a^i \otimes \nu^i)\circ F^i(\cdot, \cdot,\hatua, \cdot, e^{0,0}, e^{0,i})^{-1},  \qquad (\umu, \hatua, \ue^0 ) \in \hatuS\times\hatuA  \times \uE^0.
        \end{equation}

        \item For each player $i \in [m]$, the cost function $\hat f^i: \hatuS\times\hatuA  \to \RR$ satisfies:
        \begin{equation}
            \label{def:bar_f}
            \hat f^i(\umu, \hatua )= \int_{\fX^i\times\fA^i} f^i(x^i, a^i,\hatua)\, \hat a^i(dx^i,da^i), \qquad (\umu, \hatua) \in \hatuS\times\hatuA.
        \end{equation}

    \end{itemize}
\end{definition}

We can check that $\hat F$ is Borel measurable; see e.g. \cite[Proposition~7.29]{bertsekas_1878_book_stocha}. For the rest of the paper, we consider the MFMG lifted from our original MFTG. We will use the notations
$\hat S^i$, $\hatuS$, $\hat A^i$, $\hatuA$, $\hat U^i$,
$\Sigma$ introduced in Definition~\ref{def:MFMG}. For notational convenience, we will write
$\operatorname{pr_{x^i}}$ rather than $\operatorname{pr_{s^i}}$ when $S^i = \fX^i$.

\section{\textbf{Relations Between the Models}}\,
\label{sec:relations-models}

In this section, we establish the connection between an MFTG and its lifted MFMG. More precisely, we show that for any closed-loop Markov policy of the MFTG, there exists a corresponding Markov policy of the MFMG such that their respective value functions coincide, and vice versa. Consequently, studying the existence of a Nash equilibrium in the MFTG is equivalent to studying its existence in the MFMG.

\begin{definition}\label{def:correspondence_between_policies}
     Let $ i \in [m]$. Let $\ubpi \in \underline{\bPi}$ and $\ \hatubpi \in \hatubPi$. We say that they \defi{correspond to each other} if  for each
    $\umu\in\hatuS$, $n\ge 0$, and $i \in [m]$,
    $\hat\pi^i_n(\umu)\in\cP(\hat A^i)$ is equal to the push forward of $\PP_{\vartheta^{0,i}}$ by the map:
    $$
        \Theta^{0,i}\ni\theta^{0,i}\mapsto \mu^i(dx^i)\,\pi_n^i(x^i,\umu,\theta^{0,i})(da^i)\in \hat A^i.
    $$
\end{definition}
Note that if a level-$0$ policy profile is stationary, then every corresponding level-$1$ policy profile is stationary. Conversely, every stationary level-$1$ policy profile admits a stationary corresponding level-$0$ policy profile (see Step~2 of the proof of Theorem~\ref{thm:identical_value_function_MFTG_MFMG} below).

The next theorem plays a central role in the proof of Theorem~\ref{thm:Main_result}. It implies that, in terms of value functions, studying the MFMG lifted from the MFTG is equivalent to studying the MFTG. This theorem justifies the introduction of the notion of MFMG.

\begin{theorem}\label{thm:identical_value_function_MFTG_MFMG}
    Suppose Assumptions~\ref{assumption:atomless}~and~\ref{assumption:ufX-ufA-compact} hold. For every $\ubpi \in \underline{\bPi}$, there exists  $\hatubpi \in \hatubPi$ such that for each $i \in [m]$, $J^{i}(\ubpi) = \hat J^{i}(\hatubpi)$, and such that, for each $i \in [m]$:
    \begin{enumerate}[label=(\roman*)]
        \item for every $\bpi'^{i} \in \bPi^i$, there exists $\hat\bpi'^{i} \in \hat\bPi^i$ such that $J^{i}(\bpi'^{i}, \bpi^{-i}) = \hat J^{i}(\hat\bpi'^{i}, \hat\bpi^{-i})$;
        \item for every $\hat\bpi'^{i} \in \hat\bPi^i$, there exists $\bpi'^{i} \in \bPi^i$ such that $J^{i}(\bpi'^{i}, \bpi^{-i}) = \hat J^{i}(\hat\bpi'^{i}, \hat\bpi^{-i})$.
    \end{enumerate}
    Conversely, for every $\hatubpi \in \hatubPi$, there exists $\ubpi \in \underline{\bPi}$ such that the same properties hold. Moreover, if $\ubpi$ (resp. $\hatubpi$) is stationary, then $\hatubpi$ (resp. $\ubpi$) can be chosen stationary.
\end{theorem}

Three technical lemmas are established before giving the proof of Theorem~\ref{thm:identical_value_function_MFTG_MFMG}.
Lemma~\ref{le:reconstruction_joint_law} shows that the conditional laws $(\PP^0_{(X^i_n, \alpha^i_n)})_{1 \leq i \leq m}$ are independent when conditioned on the profile of mean fields $\umu_n = (\PP^0_{X^1_n}, \dots, \PP^0_{X^m_n})$.
Thus, once the
profile of mean fields $\umu_n$, which serves as the level-1 state, is fixed, the central player of each team $i$ selects its control independently of the others. This provides the first indication that central players can indeed be modeled as level-1 players in the MFMG framework.

\begin{lemma}
    \label{le:reconstruction_joint_law}
     Assume~\ref{assumption:atomless} holds. Let $\ubpi \in \underline{\bPi}$, and consider a pair $(\ubX,\ubalpha)$ consisting of a state-profile process and an action-profile process generated by $\ubpi$. Then:
     \begin{equation}
     \label{eq:indep_PP^0}
         \PP^0_{(X^i_n, \alpha_n^i)} \perp_{\umu_n} \PP^0_{(X^j_n, \alpha_n^j)}, \quad \PP\text{-a.s.}, \,\forall i,j \in [m] \text{ such that } i\ne j, \, \forall n\ge0,
     \end{equation}
     where $\umu_n = (\PP^0_{X^1_n}, \dots, \PP^0_{X^m_n})$ and where $\PP^0_{(X^i_n, \alpha^i_n)}$ is a regular version of $\cL\bigl( (X^i_n, \alpha^i_n) \mid\cF_n^{0}  \bigr)$.
\end{lemma}

The proof is provided in Appendix~\ref{sec:add-details}.

Next, Lemma~\ref{le:dynamics_of_conditional_distribution} shows how the evolution of the
 profile of mean-field processes $(\umu_n)_{n\in \NN}$ together with the processes of joint probability measures
$(\PP^0_{(X^i_n,\alpha^i_n)})_{\substack{i\in [m] \\ n\in \NN}}$ can be interpreted as the state process and the action processes of the lifted MFMG, when $\ubalpha$ is a profile of level-0 action processes and $\ubX$ is the corresponding profile of state processes associated with
$\ubalpha$ (see Definition~\ref{def:MFTG-state_process_from_control_process}).

\begin{lemma}
\label{le:dynamics_of_conditional_distribution}
Assume~\ref{assumption:atomless} holds. Let $\ubalpha \in \underline{\AA}$, and let $\ubX$ be the associated state process. Then:
\begin{equation}
     \label{eq:formula_dynamics_in_lemma_MFTG_MFMG}
        \bigl(\PP^0_{X^i_{n+1}}\bigr)_{1 \leq i \leq m} = \hat F\Bigl( \bigl(\PP^0_{X^i_{n}}\bigr)_{1 \leq i \leq m},\bigl(\PP^0_{(X^i_n, \alpha_n^i)}\bigr)_{1 \leq i \leq m}, \uvarepsilon^0_{n+1}\Bigr), \qquad \PP-a.s. \qquad n \geq 0.
    \end{equation}
    So
    $
     \cL\Bigl(\bigl(\PP^0_{X^i_{n+1}}\bigr)_{i} \Bigm| \bigl(\PP^0_{X^i_{n}}\bigr)_{i}, \bigl(\PP^0_{(X^i_n, \alpha_n^i)}\bigr)_{i}\Bigr) = P\Bigl(\bigl(\PP^0_{X^i_{n}}\bigr)_{i},\bigl(\PP^0_{(X^i_n, \alpha_n^i)}\bigr)_{i}\Bigr), \quad \forall n \ge 0,
    $
    where the transition kernel $P$ was defined in~\eqref{eq:MFMG_transition_kernel_from_system_func}.
\end{lemma}

The proof is provided in Appendix~\ref{sec:add-details}.

Lemma~\ref{le:identify_conditional_joint_dist_with_kernels} establishes the equality in law between level-1 state and
state-action law  processes generated by some $\hatubpi$ (see Definition~\ref{def:MFMG_admissible_state_action_processes}
), and the level-1 state and state-action processes
$\big(\bigl(\PP^0_{X^i_n}\bigr)_{\substack{i \in [m] \\ n\in \NN}},\bigl(\PP^0_{(X^i_n,\alpha^i_n)}\bigr)_{\substack{i \in [m] \\ n\in \NN}}\big)$, when $(\ubX, \ubalpha)$ is generated by some $\ubpi$ (see Definition~\ref{def:admissible_state_action_processes}), provided that $(\hatubpi, \ubpi)$ correspond to each other.

\begin{lemma}\,
    \label{le:identify_conditional_joint_dist_with_kernels}
     Assume~\ref{assumption:atomless} holds. Let $\ubalpha \in \underline{\AA}$, and let $\ubX$ be the associated state process. For every $n \ge 0$, let $\kappa_n: \hatuS \to \cP(\hatuA)$ be the Borel measurable disintegration kernel of $\cL\bigl( \bigl(\PP^0_{X^i_n}\bigr)_{1 \leq i \leq m}, \bigl(\PP^0_{(X^i_n, \alpha^i_n)}\bigr)_{1 \leq i \leq m} \bigr)$ along its first marginal.
    Then, if $(\ubzeta, \bm{\hatueta} )$ is an $(\hatuS \times \hatuA)$-valued pair of stochastic processes  which are $\FF^0$-adapted, and satisfy:
                $\uzeta_0 = \mathfrak m\bigl( \varrho_{\cP(\ufX)} (\eta,h^{\Upsilon^0}(\sU^0)) \bigr),$ $\PP-a.s.,$
                $\uzeta_{n+1} = \hat F( \uzeta_n, \hatueta_n, \uvarepsilon_{n+1}^0)$, $\PP-a.s.$,  $n \ge 0$,
            and if
                $\cL( \hatueta_n  \mid \uzeta_n ) = \kappa_n( \uzeta_n ),$ $\PP-a.s.$ $n \ge 0$,
    we have:
        \begin{equation}
            \label{eq:identify_conditional_joint_distribution}
    \cL (\uzeta_n, \hatueta_n) =\cL \Bigl(\bigl(\PP^0_{X^i_n}\bigr)_{1 \leq i \leq m}, \bigl(\PP^0_{(X^i_n, \alpha^i_n)}\bigr)_{1 \leq i \leq m} \Bigr), \qquad n \geq 0.
        \end{equation}
\end{lemma}

The proof is provided in Appendix~\ref{sec:add-details}.
We now turn to the proof of Theorem~\ref{thm:identical_value_function_MFTG_MFMG}.

\begin{proof}[Proof of Theorem~\ref{thm:identical_value_function_MFTG_MFMG}]\;
     The proof is divided in three steps.

    \textbf{Step 1, finding an appropriate $\hatubpi$ for a given $\ubpi$:} Let $\ubpi \in \underline{\bPi}$. Let $\hatubpi$ corresponding to $\ubpi$ in the sense of Definition~\ref{def:correspondence_between_policies}.
    Let $i \in [m]$, we now check the equality of the value functions $J^{i}(\ubpi)$ and $\hat J^{i}(\hatubpi)$.
    Note that if $\ubpi$ is stationary, then so is $\hatubpi$.
     Let $(\ubX,\ubalpha)$ be a pair of state and action processes generated by $\ubpi$. Then

        \begin{equation*}
            \begin{split}
                J^{i}(\ubpi)
                & = \EE \Bigl[ \sum_{n \geq 0} \gamma^n f^i \Bigl( X^i_n, \alpha^i_n,\bigl(\PP^0_{(X^j_n,\alpha^j_n)}\bigr)_{1 \leq j \leq m} \Bigr) \Bigr]
                    \\
                    &=
                    \sum_{n \geq 0} \gamma^n \EE \Bigl[ \int_{\fX^i \times \fA^i}f^i \Bigl( x^i, a^i,\bigl(\PP^0_{(X^j_n,\alpha^j_n)}\bigr)_{1 \leq j \leq m} \Bigr) \PP^0_{(X^i_n,\alpha^i_n)}(dx^i,da^i) \Bigr]
                    \\
                    &
                = \EE \Bigl[ \sum_{n \geq 0} \gamma^n \hat f^i \Bigl( \bigl(\PP^0_{X^j_n}\bigr)_{1 \leq j \leq m},\bigl(\PP^0_{(X^j_n, \alpha_n^j)}\bigr)_{1 \leq j \leq m} \Bigr) \Bigr].
            \end{split}
        \end{equation*}
        Examining the last term together with Equation~\eqref{eq:MFMG_valuef}, we would like to set
        $\umu_n = \bigl(\PP^0_{X^i_n}\bigr)_{1 \leq i \leq m}$ and $\hatua_n = \bigl(\PP^0_{(X^i_n, \alpha^i_n)}\bigr)_{1 \leq i \leq m}$, and then verify that the pair consisting of the state profile and the action profile,
        $(\ubmu, \hatuba)$, is generated by $\hatubpi$. This is indeed the case because equation \eqref{eq:level-1_gen_law_mun} is implied by Lemma~\ref{le:dynamics_of_conditional_distribution}
        , and equation \eqref{eq:level-1_gen_law} is implied by the definition of $\hatupi_n$ together with Lemma~\ref{le:reconstruction_joint_law}.

        Using the Definition~\ref{def:MFMG_VALUE_FUNCTION} of the value function $\hat J^{i}$, we have

        \begin{equation*}
            J^{i}(\ubpi)
            = \EE \Bigl[ \sum_{n \geq 0} \gamma^n \hat f^i \Bigl( \bigl(\PP^0_{X^j_n}\bigr)_{1 \leq j \leq m},\bigl(\PP^0_{(X^j_n, \alpha_n^j)}\bigr)_{1 \leq j \leq m} \Bigr) \Bigr] = \hat J^{i}(\hatubpi).
        \end{equation*}

    \textbf{Step 2, finding an appropriate $\ubpi$ for a given $\hatubpi$:} Conversely, let $\hatubpi = (\hatupi_n)_{n \geq 0} $ in $\hatubPi$. For every $n \geq 0$ and $i \in [m]$,
    $\hat \pi_n^i: \hatuS \to \cP( \hat A^i)$ is a Borel measurable map such that for every $\umu\in \hatuS$ we have  $\hat\pi_n^i(\umu)\bigl(\hat U^i(\mu^i)\bigr)=1$.
     According to the
     universal disintegration Theorem~\cite[Corollary 1.26]{Kallenberg_RM}, there exists for each $i \in [m]$ a Borel measurable probability kernel  $K^i: \fX^i \times \cP(\fX^i \times \fA^i) \times \cP(\fX^i) \to \cP(\fA^i)$ such that for every $\hat a^i\in\cP(\fX^i\times \fA^i) $ and $\mu^i \in \cP(\fX^i)$ such that $\text{pr}_{x^i}(\hat a^i)=\mu^i$, we have $\hat a^i=\mu^i\measprod K^i(\cdot,\hat a^i,\mu^i)$.
    So for every integer $n\ge 0$, $x^i\in \fX^i$,
    $\umu\in \hatuS$ and $\theta^{0,i}\in\Theta^{0,i}$, we define:

        \begin{equation}
            \label{eq:construction_closed-loop_Markov_policy_MFC}
            \pi_n^i( x^i, \umu, \theta^{0,i}) := K^i\Bigl( x^i, \varrho_{\hat A^i}\bigl(\hat \pi^i_n (\umu), h^{\Theta^{0,i}}(\theta^{0,i})\bigr),\mu^i\Bigr),
        \end{equation}
    where $\varrho_{\hat A^i}$ is the Blackwell-Dubins function of $\hat A^i$. This is licit: since $\hat\pi^i_n(\umu)\bigl(\hat U^i(\mu^i)\bigr)=1$, point~\textbf{(i)} of Lemma~\ref{le:BlackwellDubins} gives $\varrho_{\hat A^i}\bigl(\hat \pi^i_n (\umu), h^{\Theta^{0,i}}(\theta^{0,i})\bigr) \in \hat U^i(\mu^i)$, that is $\operatorname{pr_{x^i}}$ of it equals $\mu^i$, for $\PP_{\vartheta^{0,i}}$-almost every $\theta^{0,i}$.
        Because the functions $K^i$, $h^{\Theta^{0,i}}$, $\varrho_{\hat A^i}$
    are Borel measurable, so is the strategy function $\pi_n^i$ for every $n \geq 0$ and $i \in [m]$. By Definition~\ref{def:Markovian_policy}, $\bpi^i = (\pi_n^i)_{n\ge 0} \in \bPi^{i}$. If $\hatubpi$ is stationary, then the policy profile $\ubpi$ defined by~\eqref{eq:construction_closed-loop_Markov_policy_MFC} is stationary. Recall that the function $ h^{\Theta^{0,i}}$ was introduced in Section~\ref{subsec:proba_set_up}, and that $ h^{\Theta^{0,i}}(\vartheta^{0,i})$ is uniformly distributed on $[0,1]$ by construction. Notice that for every
    $\umu \in \hatuS$ and for
    $\PP_{\vartheta^{0,i}}$-almost every $\theta^{0,i} \in \Theta^{0,i}$, the definition of the universal disintegration kernel $K^i$ implies that:

        \begin{equation}
        \label{eq:pi-K-barphi}
        \begin{split}
            &\varrho_{\hat A^i}\Bigl(\hat \pi^i_n (\umu), h^{\Theta^{0,i}}(\theta^{0,i})\Bigr) (dx^i, da^i)
            \\
            &
            = \mu^i(dx^i)\, K^i \Bigl( x^i,\varrho_{\hat A^i}\bigl(\hat \pi^i_n (\umu), h^{\Theta^{0,i}}(\theta^{0,i})\bigr) , \mu^i \Bigr)(d a^i).
        \end{split}
        \end{equation}
    Using Equation~\eqref{eq:construction_closed-loop_Markov_policy_MFC}, we have:

        \begin{equation}
        \label{eq:mu-pin-rhoAbar}
            \varrho_{\hat A^i}\Bigl(\hat \pi^i_n (\umu), h^{\Theta^{0,i}}(\theta^{0,i})\Bigr)
            = \mu^i\measprod\pi_n^i( \cdot, \umu, \theta^{0,i}).
        \end{equation}
    By replacing $\theta^{0,i}$ with $\vartheta_n^{0,i}$, one can use the  Blackwell-Dubins lemma, Lemma~\ref{le:BlackwellDubins}, and obtain that the left hand side of \eqref{eq:mu-pin-rhoAbar} is a random variable with values in $\cP(\fX^i\times \fA^i)$ with distribution
    $\hat \pi^i_n (\umu)$. Moreover, the variables $\vartheta_n^{0,1},\dots,\vartheta_n^{0,m}$ being independent, these $m$ random measures are independent given $\umu$.

    \vskip 2pt
    Let us verify that $J^{ i}(\ubpi) = \hat J^{i}(\hatubpi)$ for each $i \in [m]$.
     Let $(\bzeta, \bm{\hatueta} )$ be state and action processes generated by $\hatubpi$ (see Definition~\ref{def:MFMG_admissible_state_action_processes}) and such that $\zeta_0 = \mathfrak m\bigl( \varrho_{\cP(\ufX)} (\eta,h^{\Upsilon^0}(\sU^0)) \bigr)$. Let $(\ubX,\ubalpha)$ be a pair consisting of a state-profile process and an action-profile process generated by $\ubpi$.
    Using the fact that
    $\PP^0_{(X^j_n,\alpha^j_n)}=\PP^0_{X^j_n}\measprod \pi^j_n\bigl(\cdot, (\PP^0_{X^k_n})_{1\leq k\leq m},\vartheta^{0,j}_n\bigr)$ for each $j \in [m]$, and the fact that $\uvartheta^{0}_n$ is independent of
    $(\PP^0_{X^k_n})_{1\leq k\leq m}$ by Equation~\eqref{fo:P0X_n}, we have:

        \begin{equation*}
            \begin{split}
                J^{i}(\ubpi)
                &=
                \sum_{n\ge 0}\gamma^n\EE\Bigl[\int_{\fX^i\times\fA^i}f^i\Bigl(x^i,a^i,\bigl(\PP^0_{(X^j_n,\alpha^j_n)}\bigr)_{j}\Bigr)\PP^0_{(X_n^i,\alpha_n^i)}(dx^i,da^i) \Bigr] \\
                &=
                \sum_{n\ge 0}\gamma^n\EE\Bigl[\int_{\fX^i\times\fA^i}f^i\Bigl(x^i,a^i,\bigl(\PP^0_{X^j_n}\measprod \pi^j_n(\cdot,(\PP^0_{X^k_n})_{k},\vartheta^{0,j}_n)\bigr)_{j}\Bigr) \times
                \\
                &\qquad\qquad\qquad
                \PP^0_{X_n^i}(dx^i)\,\pi_n^i\bigl(x^i,(\PP^0_{X^k_n})_{k},\vartheta^{0,i}_n\bigr)(da^i)\Bigr]
                \\
                &=
                \sum_{n\ge 0}\gamma^n\EE\Bigl[\int_{\hatuA}\int_{\fX^i\times\fA^i}f^i(x^i,a^i,\hatua)\,\hat a^i(dx^i,da^i)\,\hatupi_n\bigl((\PP^0_{X^k_n})_{k}\bigr)(d\hatua)\Bigr]
                \\
                &=
                \sum_{n\ge 0}\gamma^n\EE\Bigl[\int_{\hatuA}\hat f^i\bigl((\PP^0_{X^k_n})_{k},\hatua\bigr)\,\hatupi_n\bigl((\PP^0_{X^k_n})_{k}\bigr)(d\hatua)\Bigr].
            \end{split}
        \end{equation*}
    On the other hand:

        \begin{equation*}
            \begin{split}
                \hat J^{i}(\hatubpi)
                & = \EE \left[ \sum_{n \geq 0} \gamma^n \hat f^i ( \zeta_n, \hatueta_n ) \right]
                =  \sum_{n \geq 0} \gamma^n  \EE \left[ \int_{\hatuA} \hat f^i ( \zeta_n, \hatua ) \cL( \hatueta_n \, | \, \zeta_n)(d \hatua) \right]
                \\
                & =  \sum_{n \geq 0} \gamma^n  \EE \left[ \int_{\hatuA} \hat f^i \bigl( (\PP^0_{X^k_n})_{k}, \hatua \bigr) \hatupi_n\bigl((\PP^0_{X^k_n})_{k}\bigr)(d \hatua) \right],
            \end{split}
        \end{equation*}
    where the last equality holds by~\eqref{eq:level-1_gen_law} and Lemma~\ref{le:identify_conditional_joint_dist_with_kernels}, because $(\bzeta, \bm{\hatueta} )$ are generated by $\hatubpi $. This completes the proof of the first part of the theorem.

    \textbf{Step 3, deviations:} The computation of Step~1 only uses that $\ubpi$ and $\hatubpi$ correspond to each other, and by~\eqref{eq:mu-pin-rhoAbar} the profile $\ubpi$ built in Step~2 corresponds to $\hatubpi$. Since Definition~\ref{def:correspondence_between_policies} is componentwise, for every $\bpi'^{i} \in \bPi^i$ (resp. $\hat\bpi'^{i} \in \hat\bPi^i$), Step~1 (resp. Step~2) applied to team $i$ gives a corresponding $\hat\bpi'^{i}$ (resp. $\bpi'^{i}$), so that $(\bpi'^{i}, \bpi^{-i})$ and $(\hat\bpi'^{i}, \hat\bpi^{-i})$ correspond to each other. Step~1 then gives $J^{i}(\bpi'^{i}, \bpi^{-i}) = \hat J^{i}(\hat\bpi'^{i}, \hat\bpi^{-i})$, which proves (i) and (ii).
\end{proof}

As a direct consequence of Theorem~\ref{thm:identical_value_function_MFTG_MFMG}, we obtain the following.

\begin{corollary}
    \label{corollary:mftg_mfmg_ne_equivalence}
    Assume~\ref{assumption:atomless}~and~\ref{assumption:ufX-ufA-compact} hold. The MFTG admits a Nash equilibrium if and only if the MFMG lifted from  the MFTG admits a Nash equilibrium.
\end{corollary}

\section{\textbf{Proof of the Main Result}}
\label{section:main_result}
To establish the main result (Theorem~\ref{thm:Main_result}), we build upon the game-theoretic framework of MFMG and leverage results of Dufour and Prieto-Rumeau~\cite{dufour_2024_ne_markov_game}. However, before doing so, we introduce the notion of \emph{Young measures}, which allows us to equip the space $\hatuPi$ with a suitable topology. This step is necessary for several reasons. First, we will assume that the level-1 stochastic kernel is absolutely continuous with respect to a reference probability measure. As a result, when considering a profile of Markov policies $\hatubpi$, only the values lying in the support of the reference measure influence the dynamics. The space of \emph{Young measures} is then defined by quotienting policies according to equivalence with respect to this reference distribution. Second, some of the conditions required for the application of results Dufour and Prieto-Rumeau, notably \cite[Assumption B]{dufour_2024_ne_markov_game}, rely on this structure to properly define deviations. The following subsection aims at adapting the concepts of \cite{dufour_2024_ne_markov_game} to our setting. We also emphasize that the choice of topology on the strategy space, corresponding to topologies on stochastic kernels, is a nontrivial yet crucial issue. Rich topological structures enable the use of infinite-dimensional analysis arguments \cite{AliprantisBorder} to establish control and optimization results, and recent discussions on topologies for stochastic kernels can be found in \cite{SaldiYuksel2025KMETopology, Yuksel2024BorkarYoung}.

\subsection{The Space of Young Measures $\ucY(\lambda)$}
\label{subsec:young-measure}
Before diving into new notions, let us recall that the terminology was introduced in Paragraph~\S~\ref{subsec:measurability}. In this subsection, let us consider a reference probability measure $\lambda \in \cP(\hatuS)$ and a measurable function $q:\hatuS \times \hatuS \times \hatuA \to \RR_+$ defining a stochastic kernel $Q$ from $\hatuS \times \hatuA$ to $\hatuS$ as
$$
    Q(\sS \mid \umu, \hatua) = \int_{\sS} q(\umu', \umu, \hatua) \lambda(d\umu'), \quad \forall \sS \in \cB_{\hatuS}.
$$
We consider in $\hat \Pi^i$ the following equivalence for $\hat \pi^i, \hat \pi' \in \hat \Pi^i$:
$$
    \hat \pi \overset{\lambda}{\sim} \hat \pi' \text{ if and only if } \hat \pi(\umu) = \hat \pi'(\umu), \quad \lambda\text{-a.s.}
$$

We will denote by $\cY^i(\lambda)$ the corresponding family of equivalence classes, which will be referred to as \defi{Young measures}. The set of Young measures is equipped with the narrow topology, defined as the topology that makes the mappings
$$
    \hat \pi^i \mapsto \int_{\hatuS} \int_{\hat A^i} \hat g^i(\umu, \hat a^i) \hat \pi^i(\umu)(d\hat a^i) \lambda(d\umu),
$$
continuous for any $\hat g^i \in \mathrm{Car}(\hatuS \times \hat A^i,\RR)$ such that there exists some $G^i\in L^1(\hatuS, \cB_{\hatuS}, \lambda)$ that controls $\hat g^i$, meaning that $| \hat g^i(\umu, \hat a^i)| \, \leq  \, | G^i(\umu) |$ for every $(\umu,\hat a^i) \in \hatuS\times \hat A^i$. If $(\hat\pi_n^i)_{n\ge0} \in (\hat \Pi^i)^\NN$ converges toward $\hat\pi^i \in \hat \Pi^i$, we write $\hat \pi_n^i \xrightharpoonup[n\to \infty]{\cY^i} \hat \pi^i$.

We also define $\ucY(\lambda) = \cY^1(\lambda) \times \cdots \times \cY^m(\lambda)$ which is endowed with the product topology. We will say that two Markov strategies $\hatupi, \hatupi ' \in \hatuPi$ are in the same equivalence class of Young measure if and only if $\hat \pi^i \overset{\lambda}{\sim} \hat\pi'^{i}$ for all $i \in [m]$. We will write $\hatupi \overset{\lambda}{\sim} \hatupi '$. According to \cite{dufour_2024_ne_markov_game}, $\cY^i(\lambda)$ is a compact metric space for the defined topology. Therefore $\ucY(\lambda)$ is also a compact metric space.

Given $\hatupi \in \hatuPi$ and a function $\hat f \in \mathrm{Car}_b(\hatuS \times \hatuA, \RR)$, define the measurable function $\hat f_{\hatupi} \in L^{\infty}(\hatuS, \cB_{\hatuS}, \lambda)$ by
$$
\hat f_{\hatupi}(\umu) = \int_{\hatuA} \hat f(\umu, \hatua) \hatupi(\umu)(d\hatua).
$$

For $v \in L^{\infty}(\hatuS, \cB_{\hatuS}, \lambda) $ and $\hatupi \in \hatuPi$, we define the function $P_{\hatupi}v \in L^\infty(\hatuS, \cB_{\hatuS}, \lambda)$ by
$$
    P_{\hatupi}v(\umu) = \int_{\hatuA} \int_{\hatuS} v(\umu')q(\umu', \umu, \hatua) \lambda(d\umu') \hatupi(\umu)(d\hatua), \quad \lambda\text{-a.e. } \umu \in \hatuS.
$$

It can be shown that, for any $\hatupi \in \ucY(\lambda)$, the two elements of $L^\infty(\hatuS, \cB_{\hatuS}, \lambda)$ defined above do not depend on the representative element chosen for $\hatupi$.

\subsection{Assumptions and Main Result}

Following~\cite{CarmonaLauriereTan-2019-mfqrl}, the regularity assumptions needed for the main result are formulated on the lifted MFMG rather than on the original MFTG. We first require the continuity of the lifted one-stage costs.

\begin{assumption}[Continuity of the lifted costs]\label{assumption:continuity}
    For each $i \in [m]$, the lifted one-stage cost function $\hat f^i$ defined in~\eqref{def:bar_f} belongs to $\mathrm{Car}_b(\hatuS \times \hatuA, \RR)$.
\end{assumption}

\begin{remark}
    \label{rmk:sufficient-continuity}
    Assumption~\ref{assumption:continuity} can be checked on the original MFTG. Suppose that Assumption~\ref{assumption:ufX-ufA-compact} holds and that, for each $i \in [m]$, $f^i : \fX^i \times \fA^i \times \hatuA \to \RR$ is continuous. Then Assumption~\ref{assumption:continuity} holds. Indeed, $f^i$ is bounded by Definition~\ref{def:MFTG_model}, hence so is $\hat f^i$, and $\hat f^i$ does not depend on $\umu$ by~\eqref{def:bar_f}, so it suffices to check that $\hat f^i$ is continuous on $\hatuA$. Let $\hatua_n \xrightharpoonup[n\to\infty]{} \hatua$ in $\hatuA$. We have
    \begin{align*}
        \bigl| \hat f^i(\umu, \hatua_n) - \hat f^i(\umu, \hatua) \bigr|
        &\leq
        \int_{\fX^i\times\fA^i} \bigl|  f^i(x^i,a^i,\hatua_n) -  f^i(x^i,a^i,\hatua) \bigr| \, \hat a^i_n(dx^i,da^i)
        \\
        &\qquad +
        \Bigl| \int_{\fX^i\times\fA^i}f^i(x^i,a^i,\hatua)\, \bigl(\hat a^i_n - \hat a^i\bigr)(dx^i,da^i) \Bigr|.
    \end{align*}
    The first term converges to $0$ by the uniform continuity of $f^i$ on the compact space $\fX^i\times\fA^i\times\hatuA$, since $\hatua_n$ converges weakly to $\hatua$ and $\hat a^i_n$ is a probability measure. The second term converges to $0$ by definition of the weak convergence, $f^i(\cdot,\cdot,\hatua)$ being bounded and continuous.
\end{remark}

To apply the results of \cite{dufour_2024_ne_markov_game}, we also need to assume continuity of the cost and transition kernel in the sense of the Young measure (see Subsection~\ref{subsec:young-measure}).

\begin{assumption}[Continuity in $\ucY(\lambda)$]
    \label{assumption:continuity_sY}
    \begin{itemize}
        \item There exists a reference probability measure $\lambda \in \cP(\hatuS)$ and a measurable function $q : \hatuS \times \hatuS\times\hatuA  \to \RR_+$
       such that
    the stochastic kernel $P : \hatuS\times\hatuA  \to \hatuS$ defined in~\eqref{eq:MFMG_transition_kernel_from_system_func} is absolutely continuous with respect to $\lambda$ and admits $q$ as a density.
    That is, for any $(\umu, \hatua) \in \hatuS\times\hatuA$ and any set $B \in \cB_{\hatuS}$,
    $
        P(B \mid \umu, \hatua) = \int_B q(\umu', \umu, \hatua) \, \lambda(d\umu').
    $
    Furthermore, $q$ is such that for every $(\umu,\hatua) \in \hatuS\times\hatuA$,
    $$
    \int_{\hatuS} \bigl\vert q(\umu', \umu, \hatua_n) - q(\umu', \umu, \hatua)\bigr\vert \, \lambda(d\umu') \xrightarrow[n \to \infty]{}0,
    $$
    whenever $\hatua_n \text{ converges weakly to } \hatua \text{ in } \hatuA$.
        \item For any $v \in L^\infty(\hatuS, \cB_{\hatuS},\lambda)$, the mapping
    $
    \hatupi \mapsto P_{\hatupi}v,
    $
    defined on $\ucY(\lambda)$ and taking values in $L^\infty(\hatuS, \cB_{\hatuS},\lambda)$,
    is continuous.
        \item The mappings $\hatupi \mapsto \hat f^i_{\hatupi}$, $i \in [m]$, defined on $\ucY(\lambda)$ and taking values in $L^\infty(\hatuS, \cB_{\hatuS},\lambda)$, are continuous.
    \end{itemize}
\end{assumption}

We now provide a sufficient condition to satisfy this last assumption.

\begin{lemma}
    \label{lemma:young-measure-continuity}
    Suppose that Assumption~\ref{assumption:ufX-ufA-compact} holds, and that the first item of Assumption~\ref{assumption:continuity_sY} holds with a reference probability measure $\lambda \in \cP(\hatuS)$. Suppose moreover that
    \begin{enumerate}[label=(\roman*)]
    \item For any $1 \leq i \leq m$, there exist $m$ bounded continuous functions $(f^{i, j})_{1\leq j \leq m}$ with $f^{i, i}: \fX^i \times \fA^i \times \cP(\fX^i \times \fA^i) \to \RR$ and, for $j\ne i$, $f^{i, j}: \fX^i\times \cP(\fX^j \times \fA^j) \to \RR$, such that
    $$
    f^i(x^i, a^i, \hatua) = f^{i, i}(x^i, a^i, \hat a^i) + \sum_{\substack{1 \leq j \leq m \\ j \ne i}} f^{i, j}(x^i, \hat a^j), \qquad (x^i, a^i, \hatua) \in \fX^i \times \fA^i \times \hatuA.
    $$
     \item There exist $m$ measurable functions $(q^{i})_{1\leq i \leq m}$, with $q^{i}: \hatuS \times \hatuS\times\hat A^i \to \RR_+$, such that
    $$
    \sup_{(\umu,\hat a^i) \in \hatuS\times\hat A^i} \int_{\hatuS} q^i(\umu', \umu, \hat a^i)\, \lambda(d\umu') < \infty,
    $$
and
    $$
    P(B\mid \umu, \hatua) = \sum_{i = 1}^m \int_{B} q^i(\umu', \umu, \hat a^i) \lambda(d\umu'), \qquad (\umu, \hatua) \in \Sigma, \quad B \in \cB_{\hatuS},
    $$
    and such that, for every $(\umu,\hat a^i) \in \hatuS\times\hat A^i$ and every $i \in [m]$,
    $$
    \int_{\hatuS} \bigl\vert q^i(\umu', \umu, \hat a_n^i) - q^i(\umu', \umu, \hat a^i)\bigr\vert \, \lambda(d\umu') \xrightarrow[n \to \infty]{}0,
    $$
    whenever $\hat a_n^i \text{ converges weakly to } \hat a^i \text{ in } \hat A^i$.
\end{enumerate}
Then, Assumption~\ref{assumption:continuity_sY} holds.
\end{lemma}

The proof is given in Appendix~\ref{app:arat}. A large class of games falls within the scope of this Lemma. (i) is not particularly restrictive: the cost of the team $i$ can depend on the joint state-action law of the representative agent, but without terms coupling two different lifted actions. (ii) corresponds to the situation where all the players are perturbed, even slightly, to obtain the absolute continuity, and the game has a turn-by-turn structure, meaning that at each time step, only the agents of a specific team are allowed to take actions. In the next section, we provide a general example of a game satisfying these conditions.

Then, we have the following result (recall Definition~\ref{Nash-equilibrium-level-1}).

\begin{theorem}
    \label{thm:existence_stationary_LNE}
    Suppose that
    Assumptions~\ref{assumption:atomless},~\ref{assumption:ufX-ufA-compact},~\ref{assumption:continuity} and~\ref{assumption:continuity_sY} hold. Then the MFMG lifted from the MFTG admits a stationary Nash equilibrium.
\end{theorem}

\begin{proof}
    We wish to use \cite[Theorem 3.2, p.~14]{dufour_2024_ne_markov_game}, which ensures the existence of equilibrium for absorbing games. Although our setting is a $\gamma$-discounted model, it can be transformed into an equivalent absorbing game, as discussed in \cite[p.~18]{dufour_2024_ne_markov_game} and detailed in \cite[p.~145]{altman_1999_book_constrained_mdp}. We do not recall here the definition of an absorbing game and instead refer the interested reader to \cite{altman_1999_book_constrained_mdp}. The idea is to augment the discounted game’s state space with a cemetery point, where only one action is available, and to assign zero cost to this state. The transition kernel of the game can then be modified so that the game becomes absorbing, while ensuring that the value functions remain unchanged. We apply the cited result with reward functions $r^i:=-\hat f^i$, since maximizing $r^i$ is equivalent to minimizing the cost $\hat f^i$.

    As stated by Dufour and Prieto-Rumeau~\cite{dufour_2024_ne_markov_game}, in the discounted game with no absorption, it suffices to check that Assumption~A and Assumption~B of \cite{dufour_2024_ne_markov_game} are satisfied.

    We verify the hypotheses of~\cite[Theorem 3.2, p.~14]{dufour_2024_ne_markov_game} for the MFMG lifted from the MFTG. Remember that we suppose Assumptions~\ref{assumption:atomless},~\ref{assumption:ufX-ufA-compact},~\ref{assumption:continuity} and~\ref{assumption:continuity_sY} hold. Assumption A is composed of many sub-assumptions, but only (A3), (A4), (A5) and (A6) need to be verified, since the others concern absorbing games or games with constraint functions.

    We proceed as follows:

    $\bullet$ \textbf{(A3), The $\sigma$-algebra $\cB_{\hatuS}$ is countably generated:} For each $i \in [m]$, the individual state space $\fX^i$ is compact Polish by Assumption~\ref{assumption:ufX-ufA-compact}. Consequently, the space $\cP(\fX^i)$ of probability measures on $\fX^i$, endowed with the weak topology, is Polish, and so its Borel $\sigma$-algebra $\cB_{\cP(\fX^i)}$ is countably generated. Being a finite product of Polish spaces, $\hatuS = \cP(\fX^1)\times\cdots\times\cP(\fX^m)$ is Polish, and $\cB_{\hatuS}$ is countably generated as well.

    $\bullet$ \textbf{(A4), For each player $i \in [m]$, the action set $\hat A^i$ is compact for the weak topology, and the correspondence $\hat U^i$ is weakly measurable with nonempty compact values:} By compactness of $\fX^i \times \fA^i$ (Assumption~\ref{assumption:ufX-ufA-compact}), it follows that $\hat A^i := \cP(\fX^i \times \fA^i)$ is compact for each $i$. The control correspondence $\hat U^i$ is weakly measurable and admits closed values in $\hat A^i$ according to Remark~\ref{rmk:prx_Ui}. Since $\hat A^i$ is compact, $\hat U^i$ admits compact values. Lastly, for each $i \in [m]$, let $a^i \in \fA^i$ and define the stochastic kernel $Q^i$ from $\fX^i$ to $\fA^i$ by $Q^i(\cdot\mid x^i) = \delta_{a^i}$ for each $x^i \in \fX^i$. For any $\mu^i \in \hat S^i$, $\mu^i \measprod Q^i \in \hat U^i(\mu^i)$, and $\hat U^i$ admits non-empty values.

    $\bullet$ \textbf{(A5), For each player $i \in [m]$, we have $\hat f^i \in \mathrm{Car}_b(\hatuS \times \hatuA, \RR)$:} This is exactly Assumption~\ref{assumption:continuity}.

    $\bullet$ \textbf{(A6), There exists a measurable density function $q : \hatuS \times \hatuS \times \hatuA  \to \RR_+$  such that for each $\sS \in \cB_{\hatuS}$ and $(\umu, \hatua)\in \hatuS \times \hatuA$ we have
    $P(\sS\mid\umu, \hatua) = \int_{\sS} q(\umu',\umu,\hatua)\lambda(d\umu')$ and $$
    \lim_{n\to\infty} \int_{\hatuS}\bigl| q(\umu', \umu, \hatua_n) - q(\umu', \umu, \hatua)\bigr|\lambda(d\umu') = 0
    $$
    whenever $\hatua_n \xrightharpoonup[n\to\infty]{} \hatua$:} The first item of Assumption~\ref{assumption:continuity_sY} provides the reference measure $\lambda \in \cP(\hatuS)$ and the density $q$.

    $\bullet$ {\bfseries (B), The mappings $\hatupi \mapsto \hat f^i_{\hatupi}$ and $\hatupi \mapsto P_{\hatupi} v$, defined on $\ucY(\lambda)$ and taking values in $L^\infty(\hatuS, \cB_{\hatuS}, \lambda)$,  are continuous for any $i \in [m]$, and any $v \in L^\infty(\hatuS, \cB_{\hatuS}, \lambda)$:} These are exactly the second and third items of Assumption~\ref{assumption:continuity_sY}.

    Thus, all the required assumptions of Theorem 3.2 in \cite{dufour_2024_ne_markov_game} are satisfied.

    We verified that Theorem 3.2 of \cite{dufour_2024_ne_markov_game} applies to the MFMG lifted from the MFTG,
	$$
	(m, \fX^1, \dots, \fX^m, \fA^1, \dots, \fA^m, \hat U^1, \dots, \hat U^m, P, \hat f^1, \dots, \hat f^m, \hat\eta, \gamma).
	$$
    We conclude that it admits a stationary Nash Equilibrium.
\end{proof}

The proof of Theorem~\ref{thm:Main_result} is now straightforward: Theorem~\ref{thm:existence_stationary_LNE} gives the existence of the Nash equilibrium for the lifted MFMG, and by Corollary~\ref{corollary:mftg_mfmg_ne_equivalence}, we conclude the existence of an equilibrium for the MFTG.

\section{\textbf{Example}}
\label{section:example}

We now provide an example in which all the assumptions of our general result are satisfied. Fix an integer $K \geq 2$ and $\beta \in (0,1)$, and let
\begin{itemize}
    \item $m \in \NN^*$ be the number of teams;
    \item $\fX^i = \fX = \llbracket 0, K-1 \rrbracket$, $i = 1, \dots, m$, be the finite state space for each team;
    \item $\fA^i = \fA = \{-1, 0, 1\}$, $i = 1, \dots, m$, be the agents' action space for each team;
    \item $E^{0,0} = [m]$ be the global common noise space and $\nu^{0,0} = U([m])$ the uniform law on $[m]$;
    \item $E^{0,i} = \cP(\fX)$ be the common noise space of team $i \in [m]$, and $\nu^{0,i} = \lambda^0 = \operatorname{Dir}(1,\dots,1)$ the uniform probability measure on $\cP(\fX)$;
    \item $E^i = E = [0,1] \times \{0,1\}$ be the idiosyncratic noise space of team $i \in [m]$, and $\nu^i = U([0,1]) \otimes \operatorname{Bernoulli}(\beta)$;
    \item $F^i$ be the system function defined, for each $i \in [m]$, on $\fX \times \fA \times \hatuA \times E^i \times E^{0,0} \times E^{0,i}$ with values in $\fX$ by
    \begin{equation*}
    F^i\bigl(x, a, \hatua, (\zeta, r), i^0, e^{0,i}\bigr)
    =
    \begin{cases}
    \varrho_{\fX}(e^{0,i}, \zeta),
    & r = 1,\\[1mm]
    x + a,
    & r = 0,\ i^0 = i,\\[1mm]
    x,
    & r = 0,\ i^0 \neq i,
    \end{cases}
    \pmod{K},
    \end{equation*}

     Before giving the cost functions, we offer an interpretation of these dynamic functions. At each time step, a team is sampled randomly through $I^0 \sim \nu^{0,0}$. This team is the only one that can act during the turn. An agent of the team $I^0$ either moves according to its action $a$, or is resampled, with probability $\beta$, to another position of the map. This resampling corresponds to internal interferences among the agents of a same team. These interferences are random, and shared by all the agents of a same team: they are carried by the common noise $e^{0,i}$ of the team $i$, and they force the agent to move to a position drawn at random. The agents of the teams that are not acting at this turn remain static.

    \item For each $i \in [m]$, define the cost function $f^i: \fX \times \fA \times \hatuA \to \RR$ of team $i$  by:
    \begin{equation*}
        \begin{split}
            f^i(x,a,\hatua) = w^{i,1} |x - x^i_*| + \sum_{\substack{1 \leq j \leq m \\ j \ne i}} w^{i,2}_j \int_{\fX} |x - x^j|\prxj(\hat a^j
            )(dx^j), \\ \quad  \forall (x, a,\hatua) \in  \fX \times \fA \times \hatuA,
        \end{split}
    \end{equation*}
    where $x_*^i \in \fX$ is a target to reach by the $i$-th team, $w^{i,1}\in \RR$ indicates whether the agents wish to be close to ($w^{i,1} > 0$) or far from ($w^{i,1} < 0$) the target, and $\underline{w}^{i,2} \in \RR^m$ indicates whether the team has to get close to ($w^{i,2}_j > 0$) or far from ($w^{i,2}_j < 0$) the $j$-th team. Note that this cost function does not depend on the action $a$.
    \item $\eta \in \cP(\cP(\ufX))$ be any distribution of the initial distribution.
\end{itemize}
We take $\Upsilon^0=\Upsilon=\Theta^i=\Theta^{0,i}=[0,1]$, for every $i\in[m]$, equipped with the uniform probability measure, and construct all the noise and randomization sequences to be mutually independent. Thus, Assumption~\ref{assumption:atomless} holds.
We show, in Appendix~\ref{app:ex-verif-assumptions},
that \ref{assumption:ufX-ufA-compact}, \ref{assumption:continuity}, and \ref{assumption:continuity_sY} are satisfied. We may therefore apply Theorem~\ref{thm:Main_result} and obtain the existence of a stationary Nash equilibrium.

\section{\textbf{Conclusion}}
In this work, we studied discrete-time mean field type games (MFTGs), representing Nash equilibria between a finite number of teams, each comprising infinitely many cooperative agents.
We introduced a general probabilistic framework for MFTG, subject to both idiosyncratic and mean-field level common noise. By reformulating the original MFTG problem as a mean field Markov game (MFMG), where the state variable corresponds to the distribution of agents’ states within each team, we established an equivalence between the two formulations. This equivalence allows us to apply tools from Markov game theory with state-dependent action sets, and in particular recent existence results under regularity assumptions on the transition kernel.

Our framework is broad enough to accommodate
compact Polish state and action spaces, and we highlighted some of the technical challenges related to existence of equilibrium in such settings. We illustrated our approach with a general example, which demonstrates how common noise can induce randomization within teams and ``turn-by-turn'' dynamics. We also did not address the question of uniqueness of the Nash equilibrium, and investigating this issue under our assumptions remains another possible direction for future work. Establishing that infinite-population equilibria yield approximate Nash equilibria for the corresponding finite-agent games is another important direction for future research.
Another direction consists in studying general algorithms to obtain a stationary Nash-equilibrium policy profile.  Applications with real-world data have not yet been proposed; however, they could lead the way to new techniques from an engineering perspective.

\bibliographystyle{plainnat}
\bibliography{references-mftg-arxiv}

\clearpage
\appendix

\begingroup

\setlength{\tabcolsep}{4pt}
\newcolumntype{L}[1]{>{\raggedright\arraybackslash}p{#1\textwidth}}

\section{Index of Notation}
\label{app:notation}

This appendix collects, for reference, the notation used throughout the paper. The
tables are organized in two groups: the notation attached to the MFTG of
Section~\ref{section:MFTG_model}; and the notation attached to the MFMG of
Section~\ref{section:MFMG} and to its lifted version of
Definition~\ref{def:MFTMDP_problem}.

\subsection{Mean Field Type Game}

{\small
\begin{longtable}{L{0.13} L{0.13} L{0.49} L{0.13}}
\caption{Spaces of the MFTG.}\label{tab:notation-mftg-spaces}\\
\toprule
\textbf{Space} & \textbf{Generic element} & \textbf{Meaning} & \textbf{Defined in}\\
\midrule
\endfirsthead
\toprule
\textbf{Space} & \textbf{Generic element} & \textbf{Meaning} & \textbf{Defined in}\\
\midrule
\endhead
\bottomrule
\endlastfoot
$\fX^i$ & $x^i$ & State space of team $i$; a Borel space. & Def.~\ref{def:MFTG_model}\\
$\fA^i$ & $a^i$ & Action space of team $i$; a Borel space. & Def.~\ref{def:MFTG_model}\\
$E^i$ & $e^i$ & Idiosyncratic noise space of team $i$. & Def.~\ref{def:MFTG_model}\\
$E^{0,0}$ & $e^{0,0}$ & Global common noise space. & Def.~\ref{def:MFTG_model}\\
$E^{0,i}$ & $e^{0,i}$ & Common noise space of team $i$. & Def.~\ref{def:MFTG_model}\\
$\Upsilon^0$ & $\sU^0$ & Randomization space for the initial distribution $\mu_0$. & \S\ref{subsec:proba_set_up}\\
$\Upsilon$ & $\sU$ & Randomization space for the initial profile of states $\uX_0$. & \S\ref{subsec:proba_set_up}\\
$\Theta^i$ & $\vartheta^i$ & Private randomization space of the agents of team $i$, used to sample actions. & \S\ref{subsec:proba_set_up}\\
$\Theta^{0,i}$ & $\vartheta^{0,i}$ & Randomization space of the central player of team $i$, used to sample controls. & \S\ref{subsec:proba_set_up}\\
\end{longtable}
}

{\small
\begin{longtable}{L{0.14} L{0.21} L{0.40} L{0.13}}
\caption{Random variables and processes of the MFTG.}\label{tab:notation-mftg-rv}\\
\toprule
\textbf{Symbol} & \textbf{Values in / law} & \textbf{Meaning} & \textbf{Defined in}\\
\midrule
\endfirsthead
\toprule
\textbf{Symbol} & \textbf{Values in / law} & \textbf{Meaning} & \textbf{Defined in}\\
\midrule
\endhead
\bottomrule
\endlastfoot
$(\Omega,\cF,\PP)$ & --- & Probability space supporting all the sources of randomness. & \S\ref{subsec:proba_set_up}\\
$\varepsilon^i_{n+1}$ & $E^i$, law $\nu^i$ & Idiosyncratic noise of team $i$ at step $n+1$. & \S\ref{subsec:proba_set_up}\\
$\varepsilon^{0,0}_{n+1}$ & $E^{0,0}$, law $\nu^{0,0}$ & Global common noise. & \S\ref{subsec:proba_set_up}\\
$\varepsilon^{0,i}_{n+1}$ & $E^{0,i}$, law $\nu^{0,i}$ & Common noise of team $i$. & \S\ref{subsec:proba_set_up}\\
$\sU^0$, $\sU$ & $\Upsilon^0$, $\Upsilon$; laws $\PP_{\sU^0}$, $\PP_{\sU}$ & Randomizations of the initial distribution and of the initial states. & \S\ref{subsec:proba_set_up}\\
$\vartheta^i_n$, $\vartheta^{0,i}_n$ & $\Theta^i$, $\Theta^{0,i}$; laws $\PP_{\vartheta^i}$, $\PP_{\vartheta^{0,i}}$ & Private randomization of the agents of team $i$, and randomization of its central player, at step $n$. & \S\ref{subsec:proba_set_up}\\
$\mu_0$ & $\cP(\ufX)$, law $\eta$ & Random initial distribution, $\mu_0 = \varrho_{\cP(\ufX)}(\eta, h^{\Upsilon^0}(\sU^0))$. & Def.~\ref{def:MFTG-state_process_from_control_process}\\
$X^{\ubalpha,i}_n$ & $\fX^i$ & State of the representative agent of team $i$ at step $n$ & Def.~\ref{def:MFTG-state_process_from_control_process}\\
$\alpha^i_n$ & $\fA^i$ & Level-$0$ action of team $i$ at step $n$. & Def.~\ref{def:action-control-proc}\\
$\fa^i_n$ & $\cP(\fA^i)$ & Level-$0$ mixed control of team $i$ at step $n$. & Def.~\ref{def:action-control-proc}\\
$\PP^0_{X^{\ubalpha,i}_n}$ & $\cP(\fX^i)$ & Regular version of $\cL(X^{\ubalpha,i}_n\mid\cF^0_n)$: the mean field of team $i$ at step $n$. & \eqref{fo:P0X_n}\\
$\PP^0_{(X^{\ubalpha,i}_n,\alpha^i_n)}$ & $\cP(\fX^i\times\fA^i)$ & Regular version of $\cL((X^{\ubalpha,i}_n,\alpha^i_n)\mid\cF^0_n)$: the conditional state-action law of team $i$ at step $n$. & Def.~\ref{def:MFTG-state_process_from_control_process}\\
\end{longtable}
}

{\small
\begin{longtable}{L{0.15} L{0.60} L{0.13}}
\caption{Filtrations and $\sigma$-fields of the MFTG. All are indexed by $n\ge0$.}\label{tab:notation-mftg-filtrations}\\
\toprule
\textbf{Symbol} & \textbf{Meaning} & \textbf{Defined in}\\
\midrule
\endfirsthead
\toprule
\textbf{Symbol} & \textbf{Meaning} & \textbf{Defined in}\\
\midrule
\endhead
\bottomrule
\endlastfoot
$\cF_{\mu_0}$, $\cF_{\ux_0}$ & $\sigma$-fields of the initial distribution, $\cF_{\mu_0} = \sigma(\sU^0)$, and of the initial profile of states, $\cF_{\ux_0} = \cF_{\mu_0}\vee\sigma(\sU)$. & \S\ref{subsec:filtractions-action-controls}\\
$\cF^{\uvarepsilon}_n$, $\cF^{\varepsilon^{0,0}}_n$, $\cF^{\varepsilon^{0,i}}_n$ & Filtrations generated by the idiosyncratic noises, the global common noise, and the common noise of team $i$, up to time $n$. & \S\ref{subsec:filtractions-action-controls}\\
$\cF^{\uvarepsilon^0}_n$ & Filtration generated by all the common noises up to time $n$. & \S\ref{subsec:filtractions-action-controls}\\
$\cF^{\Theta^i}_n$, $\cF^{\uTheta}_n$ & Filtrations of the private randomizations of team $i$, resp.\ of all teams. & \S\ref{subsec:filtractions-action-controls}\\
$\cF^{\Theta^{0,i}}_n$, $\cF^{\uTheta^0}_n$ & Filtrations of the central players' randomizations of team $i$, resp.\ of all teams. & \S\ref{subsec:filtractions-action-controls}\\
$\cF^0_n$, $\FF^0$ & Common information: $\cF^0_n = \cF_{\mu_0}\vee\cF^{\uvarepsilon^0}_n\vee\cF^{\uTheta^0}_n$. This is the $\sigma$-field with respect to which the conditional laws $\PP^0$ are taken. It is the information of no decision maker. & \S\ref{subsec:filtractions-action-controls}\\
$\cG^{c,i}_n$, $\GG^{c,i}$ & Information available to the central player of team $i$ to choose a control at step $n$. & \S\ref{subsec:filtractions-action-controls}\\
$\cG^{a,i}_n$, $\GG^{a,i}$ & Information available to an agent of team $i$ to choose an action, $\cG^{a,i}_n = \cG^{c,i}_n\vee\cF^{\Theta^i}_n$. & \S\ref{subsec:filtractions-action-controls}\\
\end{longtable}
}

{\small
\begin{longtable}{L{0.11} L{0.40} L{0.26} L{0.11}}
\caption{Functions and data of the MFTG model.}\label{tab:notation-mftg-functions}\\
\toprule
\textbf{Symbol} & \textbf{Signature} & \textbf{Meaning} & \textbf{Defined in}\\
\midrule
\endfirsthead
\toprule
\textbf{Symbol} & \textbf{Signature} & \textbf{Meaning} & \textbf{Defined in}\\
\midrule
\endhead
\bottomrule
\endlastfoot
$m$ & $m\in\NN^*$ & Number of teams; teams are indexed by $i\in[m]$. & Def.~\ref{def:MFTG_model}\\
$\gamma$ & $\gamma\in(0,1)$ & Discount factor. & Def.~\ref{def:MFTG_model}\\
$\eta$ & $\eta\in\cP(\cP(\ufX))$ & Law of the random initial distribution $\mu_0$. & Def.~\ref{def:MFTG_model}\\
$F^i$ & $\fX^i \times \fA^i \times \cP(\fX^1\times\fA^1) \times \cdots \times \cP(\fX^m\times\fA^m) \times E^i \times E^{0,0} \times E^{0,i} \to \fX^i$ & System function of team $i$. & Def.~\ref{def:MFTG_model}\\
$f^i$ & $\fX^i \times \fA^i \times \cP(\fX^1\times\fA^1) \times \cdots \times \cP(\fX^m\times\fA^m) \to \RR$ & One-stage cost function of team $i$. & Def.~\ref{def:MFTG_model}\\
$h^{\Upsilon^0}$, $h^{\Upsilon}$, $h^{\Theta^i}$, $h^{\Theta^{0,i}}$ & $\Upsilon^0\to[0,1]$, etc. & Uniform coordinates of the randomization spaces. & \ref{assumption:atomless}\\
$\pi^i_n$ & $\fX^i \times \cP(\fX^1) \times \cdots \times \cP(\fX^m) \times \Theta^{0,i} \to \cP(\fA^i)$ & Closed-loop Markov strategy function of team $i$ at step $n$. & Def.~\ref{def:Markovian_policy}\\
$J^i(\ubpi)$ & $\ubPi \to \RR$ & Discounted value function of team $i$. & \eqref{eq:MFTG-J_alpha}\\
\end{longtable}
}

\subsection{Mean Field Markov Game}

{\small
\begin{longtable}{L{0.13} L{0.13} L{0.49} L{0.13}}
\caption{Spaces of the MFMG.}\label{tab:notation-mfmg-spaces}\\
\toprule
\textbf{Space} & \textbf{Generic element} & \textbf{Meaning} & \textbf{Defined in}\\
\midrule
\endfirsthead
\toprule
\textbf{Space} & \textbf{Generic element} & \textbf{Meaning} & \textbf{Defined in}\\
\midrule
\endhead
\bottomrule
\endlastfoot
$S^i$ & $s^i$ & Compact Polish space measured by the mean field of team $i$; equal to $\fX^i$ in the lifted game. & Def.~\ref{def:MFMG}\\
$\hat S^i = \cP(S^i)$ & $\mu^i$ & State space of the mean field of team $i$. & Def.~\ref{def:MFMG}\\
$A^i$ & $a^i$ & Compact Polish space of the level-$0$ actions; equal to $\fA^i$ in the lifted game. & Def.~\ref{def:MFMG}\\
$\hat A^i = \cP(S^i\times A^i)$ & $\hat a^i$ & Action space of player $i$: a joint law of state and action. & Def.~\ref{def:MFMG}\\
$\hat U^i(\mu^i)$ & $\hat a^i$ & Admissible actions of player $i$ at $\mu^i$. & Def.~\ref{def:MFMG}\\
$\Sigma$ & $(\umu,\hatua)$ & Admissible state-action pairs. All the level-$1$ processes take their values in $\Sigma$. & \eqref{eq:Gamma_bar}, Rem.~\ref{rmk:prx_Ui}\\
\end{longtable}
}

{\small
\begin{longtable}{L{0.11} L{0.40} L{0.26} L{0.11}}
\caption{Functions and data of the MFMG, and of the MFMG lifted from the MFTG.}\label{tab:notation-mfmg-functions}\\
\toprule
\textbf{Symbol} & \textbf{Signature} & \textbf{Meaning} & \textbf{Defined in}\\
\midrule
\endfirsthead
\toprule
\textbf{Symbol} & \textbf{Signature} & \textbf{Meaning} & \textbf{Defined in}\\
\midrule
\endhead
\bottomrule
\endlastfoot
$\operatorname{pr_{s^i}}$ & $\hat A^i \to \hat S^i$, and $\cP(\uS)\to\cP(S^i)$ & Marginal on the state variable. & Def.~\ref{def:MFMG}\\
$P$ & Stochastic kernel on $\hatuS$ given $\hatuS\times\hatuA$ & Transition kernel of the MFMG. & Def.~\ref{def:MFMG}\\
$\hat f^i$ & $\hatuS\times\hatuA\to\RR$ & One-stage cost function of player $i$. & Def.~\ref{def:MFMG}\\
$\hat\eta$, $\gamma$ & $\hat\eta\in\cP(\hatuS)$, $\gamma\in(0,1)$ & Law of the initial state and discount factor. & Def.~\ref{def:MFMG}\\
$\mathfrak m$ & $\cP(\ufX)\to\hatuS$ & Marginalization map $\mathfrak m(\mu) = (\operatorname{pr_{x^i}}(\mu))_{i\in[m]}$; in the lifted MFMG, $\hat\eta = \eta\circ\mathfrak m^{-1}$. & Def.~\ref{def:MFTMDP_problem}\\
$\hat F^i$ & $\hatuS\times\hatuA\times\uE^0\to\hat S^i$ & Lifted system function of team $i$. & \eqref{eq:MFMGM_Fbar_pushforward_i}\\
$\hat\pi^i$, $\hat\Pi^i$ & $\hatuS\to\cP(\hat A^i)$ with $\hat\pi^i(\umu)(\hat U^i(\mu^i)) = 1$ & Level-$1$ mixed strategy function of player $i$, and the set of such functions. & Def.~\ref{def:MFMG_POLICIES}\\
$\umu_n$ & $\hatuS$ & Level-$1$ state at step $n$. & Def.~\ref{def:MFMG_admissible_state_action_processes}\\
$\hatua_n$ & $\hatuA$ & Level-$1$ profile of actions at step $n$. & Def.~\ref{def:MFMG_admissible_state_action_processes}\\
$\hat J^i(\hatubpi)$ & $\hatubPi \to \RR$ & Discounted value function of player $i$ in the MFMG. & \eqref{eq:MFMG_valuef}\\
\end{longtable}
}

\endgroup

\section{MFTG Value Functions}
\label{app:level1-pair-determined}

The statement below is used in Section~\ref{subsection:value-function} to define the level-$0$ value functions, but its proof is more naturally written with the vocabulary of the lifted game. We therefore use throughout this appendix the notation introduced in Section~\ref{section:MFMG}: the level-$1$ state and action spaces $\hatuS$ and $\hatuA$ of Definition~\ref{def:MFMG}, and the lifted system function $\hat F$ together with the lifted cost functions $\hat f^i$ of Definition~\ref{def:MFTMDP_problem}. We also use Lemma~\ref{le:reconstruction_joint_law} and Lemma~\ref{le:dynamics_of_conditional_distribution}, both established in Section~\ref{sec:relations-models}.

\begin{lemma}
    \label{le:level1_pair_determined_by_policy}
    Let $\ubpi \in \underline{\bPi}$ and let $(\ubX, \ubalpha)$ be a pair consisting of a state-profile process and an action-profile process generated by $\ubpi$. Write $\umu_n := \bigl(\PP^0_{X^j_n}\bigr)_{j}$ and $\hatua_n := \bigl(\PP^0_{(X^j_n, \alpha^j_n)}\bigr)_{j}$. Then the law of the process $(\umu_n, \hatua_n)_{n \ge 0}$ depends only on $\ubpi$ and $\eta$. Consequently, for each $i \in [m]$,
    $$
        J^i(\ubalpha) = \sum_{n \ge 0} \gamma^n\, \EE\bigl[ \hat f^i(\umu_n, \hatua_n) \bigr]
    $$
    depends on $\ubalpha$ only through $\ubpi$, and $J^i(\ubpi) := J^i(\ubalpha)$ is well defined.
\end{lemma}

\begin{proof}
    By Definition~\ref{def:MFTG-state_process_from_control_process}, $\mu_0 = \varrho_{\cP(\ufX)}\bigl(\eta, h^{\Upsilon^0}(\sU^0)\bigr)$ is $\cF_{\mu_0}$-measurable and $\cL(\uX_0 \mid \cF_{\mu_0}) = \mu_0$; since $\uvartheta^0_0$ is independent of $\sigma(\uX_0) \vee \cF_{\mu_0}$, we get $\cL(\uX_0 \mid \cF^0_0) = \mu_0$ and therefore $\umu_0 = (\operatorname{pr_{x^i}}(\mu_0))_{1\leq i\leq m}$, a Borel measurable function of $\sU^0$.

    Let $n \ge 0$ and suppose that $\umu_n$ is a Borel measurable function of $\bigl(\sU^0, (\uvartheta^0_k)_{k \le n-1}, (\uvarepsilon^0_k)_{1 \le k \le n}\bigr)$. By~\eqref{eq:definition_action_Markov_closed_loop}, $\hatua_n$ is a Borel measurable function of $(\umu_n, \uvartheta^0_n)$, the strategy functions $\pi^i_n$ being Borel measurable. By Lemma~\ref{le:dynamics_of_conditional_distribution} and the measurability of $\hat F$, $\umu_{n+1}$ is then a Borel measurable function of $(\umu_n, \hatua_n, \uvarepsilon^0_{n+1})$. The induction gives Borel measurable maps $(\Lambda_n)_{n \ge 0}$, built from $\eta$, $\hat F$ and the strategy functions $(\pi^i_k)_{i,k}$ alone, such that
    $$
        (\umu_n, \hatua_n) = \Lambda_n\Bigl( \sU^0, (\uvartheta^0_k)_{0 \le k \le n}, (\uvarepsilon^0_k)_{1 \le k \le n} \Bigr), \qquad \PP\text{-a.s.}
    $$
    None of these maps involves $\uvarepsilon$, $\uvartheta$ or $\sU$, and the law of $\bigl(\sU^0, (\uvartheta^0_k)_{k \ge 0}, (\uvarepsilon^0_k)_{k \ge 1}\bigr)$ is prescribed by the model. Hence the law of $(\umu_n, \hatua_n)_{n \ge 0}$ is determined by $\ubpi$ and $\eta$.

    Finally, $\hatua_n$ is $\cF^0_n$-measurable and $\cL\bigl((X^i_n, \alpha^i_n) \mid \cF^0_n\bigr) = \hat a^i_n$, so conditioning on $\cF^0_n$ and using~\eqref{def:bar_f} gives $\EE\bigl[ f^i(X^i_n, \alpha^i_n, \hatua_n) \bigr] = \EE\bigl[ \hat f^i(\umu_n, \hatua_n) \bigr]$. Since $f^i$ is bounded and $\gamma \in (0,1)$, the series converges absolutely and may be summed term by term.
\end{proof}

\section{MFMG Value Functions}\label{app:MFMG_valuef}

\begin{lemma}
    \label{le:equality_law_pair_MFMG}
    Let $(\ubmu, \hatuba)$ and  $(\ubmu', \hatuba')$ be two pairs of profile of states and profile of actions generated by $\hatubpi$. Then, for all $n \ge 0$, $\cL(\umu_n) = \cL(\umu_n')$ and $ \cL(\umu_n, \hatua_n) = \cL(\umu_n', \hatua'_n)$.
\end{lemma}

\begin{proof}
    We proceed by induction, the induction hypothesis being, for $n \ge 0$,
    $$
        (H_n) : \quad \cL(\umu_n, \hatua_n) = \cL(\umu_n', \hatua'_n).
    $$

    \textbf{Initialization.} By Definition~\ref{def:MFMG_admissible_state_action_processes}, $\cL(\umu_0) = \cL(\umu_0') = \hat\eta$. Let $\phi: \hatuS \times \hatuA \to \RR$ be bounded and measurable. Using~\eqref{eq:level-1_gen_law} and then~\eqref{eq:hatupi_prod},
    \begin{align*}
        \EE\bigl[\phi(\umu_0, \hatua_0)\bigr]
        = \EE\Bigl[ \int_{\hatuA} \phi(\umu_0, \hatua) \, \hatupi_0(\umu_0)(d\hatua) \Bigr]
        &= \EE\Bigl[ \int_{\hatuA} \phi(\umu_0', \hatua) \, \hatupi_0(\umu_0')(d\hatua) \Bigr]
        \\
        &= \EE\bigl[\phi(\umu_0', \hatua_0')\bigr],
    \end{align*}
    the middle equality holding because $\umu \mapsto \int_{\hatuA} \phi(\umu, \hatua) \hatupi_0(\umu)(d\hatua)$ is a bounded measurable function of $\umu$ and $\cL(\umu_0) = \cL(\umu_0')$. Hence $(H_0)$ holds.

    \textbf{Induction step.} Let $n \ge 0$ and suppose that $(H_n)$ holds. By~\eqref{eq:level-1_gen_law_mun}, $\cL(\umu_{n+1} \mid \umu_n, \hatua_n) = P(\umu_n, \hatua_n)$ and $\cL(\umu'_{n+1} \mid \umu'_n, \hatua'_n) = P(\umu'_n, \hatua'_n)$, with the same Borel measurable kernel $P$. Together with $(H_n)$, this yields
    $$
        \cL\bigl(\umu_n, \hatua_n, \umu_{n+1}\bigr) = \cL\bigl(\umu_n', \hatua'_n, \umu'_{n+1}\bigr),
    $$
    and in particular $\cL(\umu_{n+1}) = \cL(\umu'_{n+1})$. Let now $\phi: \hatuS \times \hatuA \to \RR$ be bounded and measurable. We have:
    \begin{align*}
        \EE\bigl[\phi(\umu_{n+1}, \hatua_{n+1}) \bigr]
        &=
        \EE\bigl[ \EE \bigl( \phi(\umu_{n+1}, \hatua_{n+1}) \mid \umu_{n+1} \bigr) \bigr]
        \\
        &=
        \EE\Bigl[ \int_{\hatuA} \phi(\umu_{n+1}, \hatua) \, \cL(\hatua_{n+1}\mid\umu_{n+1})(d\hatua) \Bigr]
        \\
        &=
        \EE\Bigl[ \int_{\hatuA} \phi(\umu_{n+1}, \hatua) \prod_{i = 1}^m\hat \pi^i_{n+1}(\umu_{n+1})(d\hat a^i) \Bigr]
        \\
        &=
        \EE\Bigl[ \int_{\hatuA} \phi(\umu_{n+1}', \hatua) \prod_{i = 1}^m\hat \pi^i_{n+1}(\umu_{n+1}')(d\hat a^i) \Bigr]
        \\
        &=
        \EE\Bigl[ \int_{\hatuA} \phi(\umu_{n+1}', \hatua) \, \cL(\hatua_{n+1}'\mid\umu_{n+1}')(d\hatua) \Bigr]
        \\
        &=
        \EE\bigl[ \EE \bigl( \phi(\umu_{n+1}', \hatua_{n+1}') \mid \umu_{n+1}' \bigr) \bigr] = \EE\bigl[\phi(\umu_{n+1}', \hatua_{n+1}') \bigr],
    \end{align*}
    where in the third line we used~\eqref{eq:level-1_gen_law} together with~\eqref{eq:hatupi_prod}, in the fourth line the equality $\cL(\umu_{n+1}) = \cL(\umu'_{n+1})$ established above, and in the fifth line~\eqref{eq:level-1_gen_law} and~\eqref{eq:hatupi_prod} again. Hence $(H_{n+1})$ holds, which completes the induction.
\end{proof}

\begin{lemma}
    \label{le:MFMG_valuef}
    For each $\hatubpi \in \hatubPi$ and each player $i \in [m]$, the value function $\hat J^{i}(\hatubpi)$ given in~\eqref{eq:MFMG_valuef} is well defined.
\end{lemma}

\begin{proof}
    Let $(\ubmu,\hatuba)$ be a pair consisting of a state-profile process and an action-profile process generated by $\hatubpi$. Since $\hat f^i$ is bounded and measurable and $\gamma \in (0,1)$, the series defining $\hat J^{i}(\hatubpi)$ converges absolutely, and by linearity of the expectation it depends on $(\ubmu, \hatuba)$ only through the family of laws $\bigl(\cL(\umu_n, \hatua_n)\bigr)_{n \ge 0}$. By Lemma~\ref{le:equality_law_pair_MFMG}, this family is the same for any two pairs generated by $\hatubpi$. Hence the value does not depend upon the particular choice of the pair of state and action processes generated by $\hatubpi$, and the function is well defined.
\end{proof}

\section{Additional Details for Section~\ref{sec:relations-models}}
\label{sec:add-details}

\begin{proof}[Proof of Lemma~\ref{le:reconstruction_joint_law}]

    Let $n\ge 0$ and write $\umu_n := \bigl(\PP^0_{X^1_n}, \dots, \PP^0_{X^m_n}\bigr) \in \hatuS$. By definition, $\umu_n$ is $\sigma(\sU^0, \uvarepsilon^0_k, \uvartheta^0_{k-1}, 1 \leq k \leq n)$-measurable, hence $\cF^0_n$-measurable.

    Let $i \in [m]$. By~\eqref{eq:definition_action_Markov_closed_loop}, $\alpha^i_n$ is $\sigma(X^i_n, \umu_n, \vartheta^{0,i}_n, \vartheta^i_n)$-measurable and
    $$
        \cL\bigl(\alpha^i_n \mid X^i_n, \cF^0_n\bigr) =  \pi^i_n\bigl(X^i_n, \umu_n, \vartheta^{0,i}_n\bigr), \quad \PP\text{-a.s.},
    $$
    since $\vartheta^{0,i}_n$ is $\cF^0_n$-measurable while $\vartheta^{i}_n$ is independent of $\cF^0_n \vee \sigma(X^i_n)$. Disintegrating along $X^i_n$ (\cite[Corollary 3.6]{kallenberg2002foundations}), we obtain
    \begin{equation}
        \label{eq:app-indep-kernel-form}
        \PP^0_{(X^i_n, \alpha^i_n)} = \PP^0_{X^i_n} \measprod \pi^i_n\bigl(\cdot\,, \umu_n, \vartheta^{0,i}_n\bigr)
        = \Psi^i_n\bigl(\umu_n, \vartheta^{0,i}_n\bigr), \quad \PP\text{-a.s.},
    \end{equation}
    where $\Psi^i_n : \hatuS \times \Theta^{0,i} \to \hat A^i$ is the Borel measurable map defined by $\Psi^i_n(\umu, \theta^{0,i}) := \mu^i \measprod \pi^i_n(\cdot\,, \umu, \theta^{0,i})$, and where we used that $\PP^0_{X^i_n}$ is the $i$-th component of $\umu_n$.

    By~\eqref{eq:app-indep-kernel-form}, each $\PP^0_{(X^i_n, \alpha^i_n)}$ is a measurable function of the pair $(\umu_n, \vartheta^{0,i}_n)$. Moreover, the variables $\vartheta^{0,1}_n, \dots, \vartheta^{0,m}_n$ are independent, and $\uvartheta^{0}_n$ is independent of
    $\cF_{\mu_0}\vee\sigma(\uvarepsilon^0_k,\underline{\vartheta}^{0}_{k-1},\,1\leq k\leq n)$, hence of $\umu_n$. Consequently, conditionally on $\umu_n$, the random measures $\PP^0_{(X^1_n, \alpha^1_n)}, \dots, \PP^0_{(X^m_n, \alpha^m_n)}$ are independent, that is
    $$
        \PP^0_{(X^i_n, \alpha^i_n)} \perp_{\umu_n} \PP^0_{(X^j_n, \alpha^j_n)}, \quad \PP\text{-a.s.}, \quad  \forall i,j \in [m]\text{ such that } i \ne j, \quad \forall n\in \NN,
    $$
    which is~\eqref{eq:indep_PP^0}.
\end{proof}

\begin{proof}[Proof of Lemma~\ref{le:dynamics_of_conditional_distribution}]
    Write $\hatua_n := \bigl(\PP^0_{(X^j_n, \alpha^j_n)}\bigr)_{j}$. Fix $i \in [m]$ and let $\phi: \fX^i \to \RR$, $\psi: \uE^0 \to \RR$, $g: \Upsilon^0 \to \RR$ and $h: (\underline{\Theta}^0 \times \uE^0)^n \to \RR$ be bounded measurable functions. Set
    $$
        G := g(\sU^0)\, h\bigl((\uvartheta^0_k)_{k \leq n}, (\uvarepsilon^0_k)_{k \leq n}\bigr),
    $$
    a bounded $\cF^0_n$-measurable random variable. We have:
    \begin{align*}
        &\EE\big[\psi(\uvarepsilon^0_{n+1})\, G\, \phi(X^i_{n+1})\big]
        \\
        &=
        \EE\big[\psi(\uvarepsilon^0_{n+1})\, G\,
        \phi\bigl(F^i(X^i_n, \alpha^i_n, \hatua_n, \varepsilon^{i}_{n+1}, \varepsilon^{0,0}_{n+1}, \varepsilon^{0,i}_{n+1})\bigr)\big]
        \\
        &=
        \int_{\uE^0} \unu^0(d\ue^0)\,\psi(\ue^0)\,
        \EE \Bigg[ G  \times
        \\
        &\qquad\qquad\qquad
        \int_{\fX^i \times \fA^i}\int_{E^i} \phi\bigl(F^i(x^i, a^i, \hatua_n, e^i, e^{0,0}, e^{0,i})\bigr)\, \nu^i(de^i)\, \hatua_n^{\,i}(dx^i, da^i) \Bigg]
        \\
        &=
        \EE \Bigg[ \psi(\uvarepsilon^0_{n+1})\, G  \times
        \\
        &\qquad\qquad\qquad
        \int_{\fX^i \times \fA^i \times E^i} \phi\bigl(F^i(x^i, a^i, \hatua_n, e^i, \varepsilon^{0,0}_{n+1}, \varepsilon^{0,i}_{n+1})\bigr)\, \bigl(\hatua_n^{\,i} \otimes \nu^i\bigr)(dx^i, da^i, de^i) \Bigg]
        \\
        &=
        \EE \left[ \psi(\uvarepsilon^0_{n+1})\, G \int_{\fX^i} \phi(x^i)\, \zeta^i(dx^i) \right],
    \end{align*}
    where $\hatua_n^{\,i} = \PP^0_{(X^i_n, \alpha^i_n)}$ and
    $$
        \zeta^i := \bigl(\PP^0_{(X^i_n, \alpha^i_n)} \otimes \nu^i\bigr)\circ F^i\bigl(\cdot, \cdot, \hatua_n, \cdot, \varepsilon^{0,0}_{n+1}, \varepsilon^{0,i}_{n+1}\bigr)^{-1}
        = \hat F^i\bigl( (\PP^0_{X^j_n})_{j}, \hatua_n, \uvarepsilon^0_{n+1}\bigr)
    $$
    by~\eqref{eq:MFMGM_Fbar_pushforward_i}. In the second equality we used that $\varepsilon^i_{n+1}$ and $\uvarepsilon^0_{n+1}$ are independent of $\cF^0_n \vee \sigma(X^i_n, \alpha^i_n)$, with respective laws $\nu^i$ and $\unu^0$, together with $\cL\bigl((X^i_n, \alpha^i_n) \mid \cF^0_n\bigr) = \PP^0_{(X^i_n, \alpha^i_n)}$; in the third one we integrated $\ue^0$ back. Note that $\zeta^i$ is a random measure, since it depends on the random variables $\hatua_n$ and $\uvarepsilon^0_{n+1}$.

    As $g$, $h$ and $\psi$ range over the bounded measurable functions, the products $\psi(\uvarepsilon^0_{n+1})\,G$ generate $\cF_{\mu_0} \vee \cF^{\uvarepsilon^0}_{n+1} \vee \cF^{\uTheta^0}_{n}$. Moreover $\uvartheta^0_{n+1}$ is independent of that $\sigma$-algebra and of $(X^i_{n+1}, \zeta^i)$, so the above also determines the conditional expectations given $\cF^0_{n+1}$. Since $\zeta^i$ is $\cF^0_{n+1}$-measurable, we conclude that $\PP^0_{X^i_{n+1}} = \zeta^i$, $\PP$-a.s. As this holds for every $i \in [m]$, Equation~\eqref{eq:formula_dynamics_in_lemma_MFTG_MFMG} follows. Finally, $\uvarepsilon^0_{n+1}$ has law $\unu^0$ and is independent of $\cF^0_n$, so taking conditional expectations in~\eqref{eq:formula_dynamics_in_lemma_MFTG_MFMG} and using~\eqref{eq:MFMG_transition_kernel_from_system_func} gives the announced identity for the transition kernel $P$.
\end{proof}

\begin{proof}[Proof of Lemma~\ref{le:identify_conditional_joint_dist_with_kernels}]
    Write $\umu_n := \bigl(\PP^0_{X^i_n}\bigr)_{i}$ and $\hatua_n := \bigl(\PP^0_{(X^i_n, \alpha^i_n)}\bigr)_{i}$. We show the result by induction.

    For $n = 0$, the variable $\varrho_{\cP(\ufX)}(\eta, h^{\Upsilon^0}(\sU^0))$ has law $\eta$ by point~\textbf{(i)} of Lemma~\ref{le:BlackwellDubins}, hence $\zeta_0 = \mathfrak m\bigl(\varrho_{\cP(\ufX)}(\eta, h^{\Upsilon^0}(\sU^0))\bigr)$ and $\umu_0$ have the same law. Since moreover $\cL(\hatueta_0 \mid \zeta_0) = \kappa_0(\zeta_0)$ by assumption, and $\kappa_0$ is by definition the disintegration kernel of $\cL(\umu_0, \hatua_0)$ along its first marginal, we get $\cL(\zeta_0, \hatueta_0) = \cL(\umu_0, \hatua_0)$.

    Suppose that the property is true for some $n \ge 0$ and let $\phi:\hatuS \to \RR$ be a bounded measurable function. Since $\bzeta$ and $\bm{\hatueta}$ are $\FF^0$-adapted and $\uvarepsilon^0_{n+1}$ is independent of $\cF^0_n$ with law $\unu^0$, the triples $(\zeta_n, \hatueta_n, \uvarepsilon^0_{n+1})$ and $(\umu_n, \hatua_n, \uvarepsilon^0_{n+1})$ have the same law by the induction hypothesis. Therefore:
    \begin{align*}
        \EE\big[\phi(\zeta_{n+1})\big]
        &=
        \EE\big[\phi\bigl(\hat F(\zeta_n, \hatueta_n, \uvarepsilon^0_{n+1})\bigr)\big]
        \\
        &=
        \EE\big[\phi\bigl(\hat F(\umu_n, \hatua_n, \uvarepsilon^0_{n+1})\bigr)\big]
        \\
        &=
        \EE\big[\phi(\umu_{n+1})\big],
    \end{align*}
    where we used in the first line the assumption on $\zeta_{n+1}$, in the second line the induction hypothesis and the measurability of $\hat F$, and in the third line Lemma~\ref{le:dynamics_of_conditional_distribution}. We thus proved that $\cL(\zeta_{n+1}) = \cL(\umu_{n+1})$. Now, recall that by assumption, $(\zeta_{n+1}, \hatueta_{n+1})$ and $(\umu_{n+1}, \hatua_{n+1})$ share the same regular version $\kappa_{n+1}$ of the conditional probability of the second component given the first. We conclude that~\eqref{eq:identify_conditional_joint_distribution} holds for $n+1$ instead of $n$.
\end{proof}

\section{Proof of Lemma~\ref{lemma:young-measure-continuity}}
\label{app:arat}

\begin{proof}
	Item~(ii) gives the \emph{additive transition} decomposition on the admissible graph $\Sigma$. Since every admissible policy $\hatupi$ is supported on $\hatuU(\umu)$, only the values of $P$ on $\Sigma$ enter $P_{\hatupi}v$. Therefore, the transition-continuity argument of \cite[Corollary~3.6~(ii)(b)]{dufour_2024_ne_markov_game}, read with state space $\hatuS$, action spaces $\hat A^1, \dots, \hat A^m$ and reference measure $\lambda$, applies after restricting the action integrals to admissible actions. We refer to their paper for the proof regarding $P$.
    Item~(i), on the other hand, does not immediately give the \emph{additive reward} condition \cite[Corollary~3.6~(ii)(a)]{dufour_2024_ne_markov_game}: as will appear in~\eqref{eq:arat-hatf-decomposition} below, the lifted cost $\hat f^i$ carries cross terms depending on the two actions $\hat a^i$ and $\hat a^j$ simultaneously, and it is only on $\Sigma$, once the dependence on $\hat a^i$ has been removed, that the required additive form is recovered.

    Let $(\hatupi_n)_{n \in \NN}$ converge to $\hatupi$ in $\ucY(\lambda)$, and let $h \in L^1(\hatuS, \cB_{\hatuS}, \lambda)$ be a test function. Since $\ucY(\lambda)$ carries the product topology, we have $\hat \pi^i_n \xrightharpoonup[n\to\infty]{\cY^i} \hat \pi^i$ for every $i \in [m]$.

    For each $i \in [m]$, we write
    $
        \Gamma^i := \bigl\{ (\umu, \hat a^i) \in \hatuS \times \hat A^i \; \big\vert \; \hat a^i \in \hat U^i(\mu^i) \bigr\},
    $
    and we recall from Definition~\ref{def:MFMG_POLICIES} that $\hat \pi^i(\umu)\bigl(\hat U^i(\mu^i)\bigr) = 1$ for every $\umu \in \hatuS$ and every $\hat \pi^i \in \hat \Pi^i$, and from~\eqref{eq:hatupi_prod}. Consequently, $\hatupi(\umu)$ is carried by $\hatuU(\umu)$, so that only the values on $\Sigma$ of the integrands below are relevant. Finally, by Assumption~\ref{assumption:ufX-ufA-compact} the spaces $\fX^i \times \fA^i$ are compact, hence so are the spaces $\hat A^i = \cP(\fX^i \times \fA^i)$ and $\hatuS$.

    \medskip
    Fix $i \in [m]$. Injecting the decomposition of item~(i) into~\eqref{def:bar_f} gives
    \begin{equation}
        \label{eq:arat-hatf-decomposition}
        \hat f^i(\umu, \hatua) = \hat f^{i, i}(\hat a^i) + \sum_{\substack{1 \leq j \leq m \\ j \ne i}} \hat f^{i, j}(\hat a^i, \hat a^j), \qquad (\umu, \hatua) \in \hatuS \times \hatuA,
    \end{equation}
    where, for $j \ne i$,
    \begin{align*}
        &\hat f^{i, i}(\hat a^i) := \int_{\fX^i\times\fA^i} f^{i, i}(x^i, a^i, \hat a^i) \, \hat a^i(dx^i, da^i),
        \\
        &\hat f^{i, j}(\hat a^i, \hat a^j) := \int_{\fX^i} f^{i, j}(x^i, \hat a^j) \, \prxi(\hat a^i)(dx^i).
    \end{align*}
    We also introduce
    $$
        \tilde f^{i, j}(\umu, \hat a^j) := \int_{\fX^i} f^{i, j}(x^i, \hat a^j) \, \mu^i(dx^i), \qquad (\umu, \hat a^j) \in \hatuS \times \hat A^j .
    $$
    Since $\hat a^i \in \hat U^i(\mu^i)$ if and only if $\prxi(\hat a^i) = \mu^i$, we obtain
    \begin{equation}
        \label{eq:arat-substitution}
        \hat f^{i, j}(\hat a^i, \hat a^j) = \tilde f^{i, j}(\umu, \hat a^j), \qquad (\umu, \hat a^i) \in \Gamma^i, \quad \hat a^j \in \hat A^j,
    \end{equation}
    whose effect is to remove the dependence on $\hat a^i$ from the cross terms.

    Let us check that $\hat f^{i, i}$ and $\tilde f^{i, j}$ are continuous and bounded. The function $f^{i, i}$ is continuous on the compact set $\fX^i \times \fA^i \times \hat A^i$, hence bounded and uniformly continuous; thus, if $\hat a^i_n$ converges weakly to $\hat a^i$,
    \begin{align*}
        \bigl\vert \hat f^{i, i}(\hat a^i_n) - \hat f^{i, i}(\hat a^i) \bigr\vert
        &\leq \sup_{(x^i, a^i)} \bigl\vert f^{i, i}(x^i, a^i, \hat a^i_n) - f^{i, i}(x^i, a^i, \hat a^i) \bigr\vert
        \\
        &\qquad + \Bigl\vert \int f^{i, i}(\cdot, \cdot, \hat a^i) \, d\hat a^i_n - \int f^{i, i}(\cdot, \cdot, \hat a^i) \, d\hat a^i \Bigr\vert,
    \end{align*}
    where the first term vanishes by uniform continuity and the second by weak convergence. The same argument, applied to $f^{i, j}$ on the compact set $\fX^i \times \hat A^j$, shows that $\tilde f^{i, j}$ is continuous on $\hatuS \times \hat A^j$ and bounded by some constant $C \in \RR_+$, in particular $\tilde f^{i, j} \in \mathrm{Car}_b(\hatuS \times \hat A^j, \RR)$.

    Integrating~\eqref{eq:arat-hatf-decomposition} against $\hatupi(\umu)$, using~\eqref{eq:hatupi_prod}, then substituting~\eqref{eq:arat-substitution} (because $\hat \pi^i(\umu)$ is carried by $\hat U^i(\mu^i)$) and finally integrating out the variable $\hat a^i$, on which $\tilde f^{i,j}$ does not depend, we get, for every $\umu \in \hatuS$,
    $$
        \hat f^i_{\hatupi}(\umu)
        = \int_{\hat A^i} \hat f^{i, i}(\hat a^i) \, \hat \pi^i(\umu)(d\hat a^i)
        + \sum_{\substack{1 \leq j \leq m \\ j \ne i}} \int_{\hat A^j} \tilde f^{i, j}(\umu, \hat a^j) \, \hat \pi^j(\umu)(d\hat a^j).
    $$
    Therefore
    \begin{align*}
        \int_{\hatuS} h(\umu) \, \hat f^i_{\hatupi_n}(\umu) \, \lambda(d\umu)
        &= \int_{\hatuS} \int_{\hat A^i} \hat g^{i}(\umu, \hat a^i) \, \hat \pi_n^i(\umu)(d\hat a^i) \, \lambda(d\umu)
        \\
        &\qquad + \sum_{\substack{1 \leq j \leq m \\ j \ne i}} \int_{\hatuS} \int_{\hat A^j} \hat g^{i, j}(\umu, \hat a^j) \, \hat \pi_n^j(\umu)(d\hat a^j) \, \lambda(d\umu),
    \end{align*}
    where $\hat g^{i}(\umu, \hat a^i) := h(\umu) \hat f^{i, i}(\hat a^i)$ and $\hat g^{i, j}(\umu, \hat a^j) := h(\umu) \tilde f^{i, j}(\umu, \hat a^j)$. Both are Carathéodory functions, and both are controlled by $C \vert h \vert \in L^1(\hatuS, \cB_{\hatuS}, \lambda)$. By the definition of the narrow topology of $\cY^i(\lambda)$ and $\cY^j(\lambda)$ recalled in Subsection~\ref{subsec:young-measure}, each of these $m$ terms converges to the corresponding term written with $\hatupi$. Hence $\int h \, \hat f^i_{\hatupi_n} \, d\lambda \to \int h \, \hat f^i_{\hatupi} \, d\lambda$ for every $h \in L^1$, which is the announced continuity.

\end{proof}

\section{Verification of the Assumptions for the Example~\ref{section:example}}

\label{app:ex-verif-assumptions}

Throughout, we identify $\cP(\fX)$ with the simplex $\{y \in \RR_+^{K} : \sum_{x \in \fX} y(x) = 1\}$, which is a compact convex subset of a $(K-1)$-dimensional affine subspace of $\RR^K$; $\lambda^0 = \operatorname{Dir}(1,\dots,1)$ is the normalized $(K-1)$-dimensional Lebesgue measure on it. Since $\fX \times \fA$ is finite, the weak topology on $\cP(\fX \times \fA)$ coincides with the one induced by the norm $\Vert \nu \Vert := \sum_{(x,a) \in \fX \times \fA} |\nu(\{(x,a)\})|$, and similarly on $\cP(\fX)$.

We introduce the following notation:
\begin{itemize}
    \item $s : \fX \times \fA \to \fX$, $s(x,a) := x + a \pmod K$, and $\theta : \cP(\fX \times \fA) \to \cP(\fX)$, $\theta(\hat a) := \hat a \circ s^{-1}$, the law of the position after the move;
    \item for $c \in \cP(\fX)$, the affine contraction $T_c : \cP(\fX) \to \cP(\fX)$, $T_c(p) := \beta p + (1-\beta)c$, and its range $\Delta_\beta(c) := T_c(\cP(\fX))$.
\end{itemize}
Note that $\theta$ is linear and $1$-Lipschitz, since $\Vert \theta(\hat a) - \theta(\hat b)\Vert \le \Vert \hat a - \hat b \Vert$ for all $\hat a, \hat b \in \cP(\fX \times \fA)$.

We first record three elementary facts about the sets $\Delta_\beta(c)$.

\begin{lemma}
    \label{le:example-simplex-geometry}
    For every $c, c' \in \cP(\fX)$:
    \begin{enumerate}[label=(\roman*)]
        \item $\Delta_\beta(c) = \bigl\{ y \in \cP(\fX) : y(x) \ge (1-\beta)c(x), \ \forall x \in \fX \bigr\}$;
        \item $\lambda^0 \circ T_c^{-1}$ is absolutely continuous with respect to $\lambda^0$, with density $\beta^{-(K-1)} \mathds{1}_{\Delta_\beta(c)}$; in particular $\lambda^0\bigl(\Delta_\beta(c)\bigr) = \beta^{K-1}$;
        \item $\displaystyle \lambda^0\bigl(\Delta_\beta(c) \, \triangle \, \Delta_\beta(c')\bigr) \; \le \; (K-1)(1-\beta) \sum_{x \in \fX} \bigl| c(x) - c'(x) \bigr|$.
    \end{enumerate}
\end{lemma}

\begin{proof}
    (i) If $y = T_c(p)$ with $p \in \cP(\fX)$, then $y - (1-\beta)c = \beta p \ge 0$. Conversely, if $y \in \cP(\fX)$ satisfies $y \ge (1-\beta)c$ pointwise, set $p := \beta^{-1}\bigl(y - (1-\beta)c\bigr)$; then $p \ge 0$ and $\sum_x p(x) = \beta^{-1}(1 - (1-\beta)) = 1$, so $p \in \cP(\fX)$ and $y = T_c(p)$.

    (ii) $T_c$ is an affine bijection from the affine hull of $\cP(\fX)$ onto itself, with linear part $\beta \operatorname{Id}$; its Jacobian determinant, computed in that $(K-1)$-dimensional affine subspace, is therefore $\beta^{K-1}$. Consequently, for every $B \in \cB_{\cP(\fX)}$,
    $$
        \bigl(\lambda^0 \circ T_c^{-1}\bigr)(B) = \lambda^0\bigl(T_c^{-1}(B \cap \Delta_\beta(c))\bigr) = \beta^{-(K-1)} \lambda^0\bigl(B \cap \Delta_\beta(c)\bigr).
    $$
    Taking $B = \cP(\fX)$ gives $\lambda^0(\Delta_\beta(c)) = \beta^{K-1}$.

    (iii) By (i), if $y \in \Delta_\beta(c) \setminus \Delta_\beta(c')$ then there exists $x \in \fX$ with $(1-\beta)c(x) \le y(x) < (1-\beta)c'(x)$; such an $x$ satisfies $c'(x) > c(x)$. Hence
    $$
        \Delta_\beta(c) \setminus \Delta_\beta(c') \subseteq \bigcup_{x \in \fX} \Bigl\{ y \in \cP(\fX) : y(x) \in \bigl[(1-\beta)c(x), \, (1-\beta)c'(x)\bigr) \Bigr\}.
    $$
    Under $\lambda^0 = \operatorname{Dir}(1,\dots,1)$, each coordinate $y(x)$ has law $\operatorname{Beta}(1, K-1)$, whose density $t \mapsto (K-1)(1-t)^{K-2}$ is bounded by $K-1$ on $[0,1]$. Therefore
    $$
        \lambda^0\bigl(\Delta_\beta(c) \setminus \Delta_\beta(c')\bigr) \le (K-1)(1-\beta) \sum_{x \in \fX} \bigl(c'(x) - c(x)\bigr)^+.
    $$
    Exchanging the roles of $c$ and $c'$ and summing yields (iii).
\end{proof}

\medskip
\textbf{Assumption~\ref{assumption:atomless}.} The Assumption is straightforward.

\medskip
\textbf{Assumption~\ref{assumption:ufX-ufA-compact}.} The spaces $\fX^i = \fX$ and $\fA^i = \fA$ are finite, hence compact.

\medskip
\textbf{Assumption~\ref{assumption:continuity}.} Each $f^i$ is continuous, so Assumption~\ref{assumption:continuity} follows from Remark~\ref{rmk:sufficient-continuity}.

\paragraph{The lifted dynamics.} We now compute $\hat F$ and $P$. Fix $(\umu,\hatua) \in \hatuS \times \hatuA$ and $\ue^0 = \bigl(i^0, e^{0,1},\dots,e^{0,m}\bigr) \in \uE^0$. Let $(X,\alpha) \sim \hat a^i$ and $(\zeta, r) \sim \nu^i$ be independent. By~\eqref{eq:MFMGM_Fbar_pushforward_i}, $\hat F^i(\umu,\hatua,\ue^0)$ is the law of $F^i\bigl(X,\alpha,\hatua,(\zeta,r),i^0,e^{0,i}\bigr)$. Decomposing over the three branches:
\begin{itemize}
    \item on $\{r = 1\}$, which has probability $\beta$, the value is $\varrho_{\fX}(e^{0,i},\zeta)$ with $\zeta \sim U([0,1])$ independent of $(X,\alpha)$, so by the first point of Lemma~\ref{le:BlackwellDubins} its conditional law is $e^{0,i}$;
    \item on $\{r = 0\} \cap \{i^0 = i\}$, the value is $s(X,\alpha)$, of law $\theta(\hat a^i)$;
    \item on $\{r = 0\} \cap \{i^0 \ne i\}$, the value is $X$, of law $\prxi(\hat a^i)$.
\end{itemize}
Setting, for $k \in [m]$,
\begin{equation}
    \label{eq:example-cik}
    c^i_k(\hat a^i) :=
    \begin{cases}
        \theta(\hat a^i), & k = i, \\
        \prxi(\hat a^i), & k \ne i,
    \end{cases}
\end{equation}
we obtain
\begin{equation}
    \label{eq:example-barF}
    \hat F^i(\umu, \hatua, \ue^0) = \beta \, e^{0,i} + (1-\beta) \, c^i_{i^0}(\hat a^i) = T_{c^i_{i^0}(\hat a^i)}\bigl(e^{0,i}\bigr), \qquad i \in [m].
\end{equation}
Under $\unu^0 = \nu^{0,0} \otimes \bigotimes_{i=1}^m \nu^{0,i} = U([m]) \otimes (\lambda^0)^{\otimes m}$, the index $i^0$ and the measures $e^{0,1},\dots,e^{0,m}$ are independent, and conditionally on $\{i^0 = k\}$ the $m$ components of~\eqref{eq:example-barF} are independent. Therefore, by~\eqref{eq:MFMG_transition_kernel_from_system_func},
\begin{equation}
    \label{eq:example-P}
    P(\umu,\hatua) = \frac{1}{m} \sum_{k=1}^m \bigotimes_{i=1}^m \Bigl( \lambda^0 \circ T_{c^i_k(\hat a^i)}^{-1} \Bigr).
\end{equation}

\medskip
\textbf{Assumption~\ref{assumption:continuity_sY}.} We check the first item of Assumption~\ref{assumption:continuity_sY} and the two conditions of Lemma~\ref{lemma:young-measure-continuity} with the reference measure
$$
    \lambda := (\lambda^0)^{\otimes m} \in \cP(\hatuS), \qquad \text{recalling that } \hatuS = \cP(\fX)^m .
$$

\emph{First item of Assumption~\ref{assumption:continuity_sY}.} By~\eqref{eq:example-P} and Lemma~\ref{le:example-simplex-geometry}(ii), for every $(\umu,\hatua) \in \hatuS\times\hatuA$, the measure $P(\umu,\hatua)$ is absolutely continuous with respect to $\lambda$, with density
$$
    q(\umu',\umu,\hatua) := \frac{1}{m \, \beta^{m(K-1)}} \sum_{k=1}^m \prod_{i=1}^m \mathds{1}_{\Delta_\beta(c^i_k(\hat a^i))}\bigl(\mu'^{\,i}\bigr),
$$
which is Borel measurable by Lemma~\ref{le:example-simplex-geometry}(i). Let $(\hatua_n)_{n\in\NN}$ converge weakly to $\hatua$ in $\hatuA$. Using $\bigl|\prod_i \mathds{1}_{A_i} - \prod_i \mathds{1}_{B_i}\bigr| \le \sum_i \bigl|\mathds{1}_{A_i} - \mathds{1}_{B_i}\bigr|$, Fubini's theorem and Lemma~\ref{le:example-simplex-geometry}(iii), we get
\begin{align*}
    \int_{\hatuS} \bigl| q(\umu',\umu,\hatua_n) - q(\umu',\umu,\hatua) \bigr| \, \lambda(d\umu')
    &\le \frac{(K-1)(1-\beta)}{m \, \beta^{m(K-1)}} \sum_{k=1}^m \sum_{i=1}^m \bigl\Vert c^i_k(\hat a^i_n) - c^i_k(\hat a^i) \bigr\Vert
    \\
    &\le \frac{(K-1)(1-\beta)}{\beta^{m(K-1)}} \sum_{i=1}^m \bigl\Vert \hat a^i_n - \hat a^i \bigr\Vert,
\end{align*}
since $\theta$ and $\prxi$ are $1$-Lipschitz for $\Vert\cdot\Vert$. The right-hand side converges to $0$, because $\fX\times\fA$ is finite, so that weak convergence in $\cP(\fX \times \fA)$ is equivalent to convergence in the norm $\Vert \cdot \Vert$.

\emph{Condition (i).} It suffices to take, for $(x,a,\hatua) \in \fX \times \fA \times \hatuA$,
$$
    f^{i,i}(x,a,\hat a^i) := w^{i,1}|x - x^i_*|, \qquad f^{i,j}(x, \hat a^j) := w^{i,2}_j \int_{\fX} |x - x^j| \, \prxj(\hat a^j)(dx^j), \quad j \ne i,
$$

so that $f^i(x,a,\hatua) = f^{i,i}(x,a,\hat a^i) + \sum_{j \ne i} f^{i,j}(x,\hat a^j)$.

\emph{Condition (ii).} Let $(\umu,\hatua) \in \Sigma$, so that $\prxi(\hat a^i) = \mu^i$ for every $i \in [m]$. Let for $(\umu',\umu,\hat a^i) \in \hatuS \times \hatuS \times \hat A^i$,
\begin{equation}
    \label{eq:example-qi}
    q^i(\umu',\umu,\hat a^i) := \frac{1}{m \, \beta^{m(K-1)}} \; \mathds{1}_{\Delta_\beta(\theta(\hat a^i))}\bigl(\mu'^{\,i}\bigr) \prod_{j \ne i} \mathds{1}_{\Delta_\beta(\mu^j)}\bigl(\mu'^{\,j}\bigr),
\end{equation}
which is Borel measurable, since by Lemma~\ref{le:example-simplex-geometry}(i) the set $\{(\mu',\hat a) : \mu'(x) \ge (1-\beta)\theta(\hat a)(x), \ \forall x \in \fX\}$ is closed. Using Equation~\eqref{eq:example-cik} and Lemma~\ref{le:example-simplex-geometry}(ii) and Equation~\eqref{eq:example-P}, for every $(\umu,\hatua) \in \Sigma$ and every $B \in \cB_{\hatuS}$,
$$
    P(B \mid \umu, \hatua) = \sum_{i=1}^m \int_B q^i(\umu',\umu,\hat a^i) \, \lambda(d\umu').
$$
By Lemma~\ref{le:example-simplex-geometry}(ii), $\int_{\hatuS} q^i(\umu',\umu,\hat a^i) \, \lambda(d\umu') = \tfrac{1}{m}$ for every $(\umu,\hat a^i) \in \hatuS\times\hat A^i$ and every $i \in [m]$, so the bound required in condition~(ii) holds.

It remains to establish the continuity in $\hat a^i$. Let $(\hat a^i_n)_{n \in \NN}$ converge weakly to $\hat a^i$ in $\hat A^i$. The factors indexed by $j \ne i$ in~\eqref{eq:example-qi} do not depend on $\hat a^i$, and each contributes $\lambda^0(\Delta_\beta(\mu^j)) = \beta^{K-1}$; hence, by Fubini's theorem and Lemma~\ref{le:example-simplex-geometry}(iii),
\begin{align*}
    \int_{\hatuS} \bigl| q^i(\umu',\umu,\hat a^i_n) - q^i(\umu',\umu,\hat a^i) \bigr| \, \lambda(d\umu')
    &= \frac{\beta^{(m-1)(K-1)}}{m \, \beta^{m(K-1)}} \; \lambda^0\Bigl( \Delta_\beta\bigl(\theta(\hat a^i_n)\bigr) \, \triangle \, \Delta_\beta\bigl(\theta(\hat a^i)\bigr) \Bigr)
    \\
    &\le \frac{(K-1)(1-\beta)}{m \, \beta^{K-1}} \; \bigl\Vert \theta(\hat a^i_n) - \theta(\hat a^i) \bigr\Vert
    \\
    &\le \frac{(K-1)(1-\beta)}{m \, \beta^{K-1}} \; \bigl\Vert \hat a^i_n - \hat a^i \bigr\Vert \xrightarrow[n \to \infty]{} 0,
\end{align*}

the last convergence holding because $\fX \times \fA$ is finite, so that weak convergence in $\cP(\fX \times \fA)$ is equivalent to convergence in the norm $\Vert \cdot \Vert$.

Both conditions of Lemma~\ref{lemma:young-measure-continuity} are therefore satisfied, and Assumption~\ref{assumption:continuity_sY} holds with reference measure $\lambda = (\lambda^0)^{\otimes m}$.

\medskip
\textbf{Conclusion.} Assumptions~\ref{assumption:atomless},~\ref{assumption:ufX-ufA-compact},~\ref{assumption:continuity} and~\ref{assumption:continuity_sY} are satisfied, so Theorem~\ref{thm:Main_result} applies.

\end{document}